\documentclass{amsart}
\usepackage[margin=1in]{geometry}
\usepackage{amsfonts,amssymb,amsmath,amsthm}
\usepackage{url}
\usepackage{enumerate}
\usepackage[all]{xypic}

\usepackage{mathrsfs}
\usepackage{bbold}

\newtheorem{theorem}{Theorem}[section]

\newtheorem{question}[theorem]{Questions}
\newtheorem{lemma}[theorem]{Lemma}
\newtheorem{proposition}[theorem]{Proposition}
\newtheorem{corollary}[theorem]{Corollary}
\theoremstyle{definition}

\newtheorem{definition}[theorem]{Definition}
\numberwithin{equation}{section}

\newcommand{\fin}{$\mathbf{Fin}\ $}
\renewcommand{\r}{$\mathrm{R}^n$}
\newcommand{\ra}[1]{$\mathrm{R}^{#1}$}
\newcommand{\hb}[1]{$\mathrm{H}_{#1}^{\mathrm{B}}$}
\newcommand{\hd}[1]{$\mathrm{H}_{#1}^{\mathrm{D}}$}
\newcommand{\hi}[1]{$\mathrm{H}_{#1}$}

\newcommand{\zfc}{\mathnormal{\mathsf{ZFC}}}
\newcommand{\zf}{\mathnormal{\mathsf{ZF}}}
\newcommand{\zfa}{\mathnormal{\mathsf{ZFA}}}
\newcommand{\ac}{\mathnormal{\mathsf{AC}}}
\newcommand{\lo}{\mathnormal{\mathsf{LO}}}
\newcommand{\bpi}{\mathnormal{\mathsf{BPI}}}
\newcommand{\fs}{\mathop{\mathrm{FS}}\nolimits}
\newcommand{\fu}{\mathop{\mathrm{FU}}}
\newcommand{\hs}{\mathnormal{\mathbf{HS}}}

\DeclareMathOperator{\dom}{dom}
\DeclareMathOperator{\ran}{ran}

\DeclareMathOperator{\aut}{Aut}

\author[J. Brot]{Joshua Brot}
\address{
Department of Mathematics\\
University of Michigan\\
2074 East Hall, 530 Church Street \\
Ann Arbor, MI 48109-1043, U.S.A.}
\email{jbrot@umich.edu}

\author[M. Cao]{Mengyang Cao}
\address{
Department of Mathematics\\
University of Michigan\\
2074 East Hall, 530 Church Street\\
Ann Arbor, MI 48109-1043, U.S.A.}
\curraddr{
Northwestern University\\
633 Clark St\\
Evanston, IL 60208, U.S.A.
}
\email{mengyangcao2020@u.northwestern.edu}

\author[D. Fern\'andez]{David Fern\'andez-Bret\'on}
\address{
Department of Mathematics, University of Michigan; 
Kurt G\"odel Research Center for Mathematical Logic, University of Vienna; and Departamento de Matem\'aticas, Centro de Investigaci\'on y Estudios Avanzados.
}
\curraddr{
Instituto de Matem\'aticas\\
Universidad Nacional Aut\'onoma de M\'exico\\
\'Area de la Investigaci\'on Cient\'{\i}fica, Circuito Exterior, Ciudad Universitaria, Coyoac\'an, 04510,\\
CDMX, Mexico\\
}
\email{djfernandez@im.unam.mx}
\urladdr{http://homepage.univie.ac.at/david.fernandez-breton/}

\begin{document}

\title[Choiceless Ramsey finiteness classes]{Finiteness classes arising from Ramsey-theoretic statements in set theory without choice}

\begin{abstract}
We investigate infinite sets that witness the failure of certain Ramsey-theoretic statements, such as Ramsey's or (appropriately phrased) Hindman's theorem; such sets may exist if one does not assume the Axiom of Choice. We obtain very precise information as to where such sets are located within the hierarchy of infinite Dedekind-finite sets.
\end{abstract}

\maketitle

\section{Introduction}

A very interesting line of research in choiceless set theory consists of exploring the relations between the various different ways of expressing finiteness of a set. The starting point for this vein of research is the observation that Dedekind's definition of an infinite set~\cite[Definition 64, p. 63]{dedekind}, which in normal circumstances (i.e. when one assumes the Axiom of Choice, which will henceforth be denoted by $\ac$) is equivalent to the ``standard'' definition, is no longer equivalent to it if one drops $\ac$. We now proceed to state both definitions; in this paper, the notation $X\approx Y$ will denote that $X$ is \textbf{equipotent} to $Y$, i.e., that there is a bijective function between $X$ and $Y$.

\begin{definition}\label{def-findfin}
Let $X$ be a set.
\begin{enumerate}
\item We say that $X$ is \textbf{finite} if there exists an $n\in\omega$ such that $n\approx X$.
\item We say that $X$ is \textbf{D-finite}\footnote{The ``D'' here stands, of course, for Dedekind; it is also common in the literature to call D-finite sets Dedekind-finite.} if there is no proper subset $Y\subsetneq X$ such that $Y\approx X$; equivalently, $X$ is D-finite if every function $f:X\longrightarrow X$ that is injective is also surjective.
\item We say that $X$ is \textbf{infinite} or \textbf{D-infinite}, respectively, if $X$ is not finite or not D-finite, respectively.
\end{enumerate}
\end{definition}

Dedekind's idea to define finiteness as in Definition~\ref{def-findfin}~(2) arose from the very old observation, sometimes attributed to Galileo, that the set of natural numbers is in bijection with one of its proper subsets. In $\zf$ (that is, assuming all of the usual axioms of set theory ---the Zermelo--Fraenkel axioms---, except for $\ac$), it is known that a set $X$ is D-infinite if and only if it has a countable subset, that is, a subset $Y\subseteq X$ such that $\omega\approx Y$ (equivalently, $X$ is D-infinite if and only if there exists an injective function $f:\omega\longrightarrow X$). Also in $\zf$, it is very easy to see that every finite set is D-finite, and assuming $\ac$, the converse implication holds too. Such implication, however, is not provable in $\zf$ alone, as shown by Cohen in his seminal work~\cite[Chapter IV \S 9]{cohen-stcontinuum} where he first introduced the technique of forcing.

Nowadays, there is extensive knowledge of many different models of $\zf$ containing sets that are infinite but at the same time D-finite. Furthermore, prompted by a seminal paper of Tarski~\cite{tarski}, numerous other authors~\cite{levy-independencefiniteness,truss-classesofdedekind,spisiak-vojtas,howard-yorke,degen,goldstern-stronglyamorphous,herrlich-choice} have continued to investigate the various other possible definitions of finiteness, all of which are equivalent under $\zfc$, but which may be different in models of $\zf$; we now have abundant information about the implication relations between many of these possible definitions. Furthermore, in recent times other authors~\cite{degen,herrlich-finiteinfinite,herrlich-howard-tachtsis}, have shifted from studying multiple isolated definitions of finiteness, to studying a general notion of what an acceptable ``definition of finiteness'' might be, resulting in what is now known as a {\it finiteness class}. The general definition of a finiteness class, as stated below, is due to Herrlich~\cite[Definition 6]{herrlich-finiteinfinite}, and is equivalent to the negation of a {\it notion of infinity} as proposed by Degen~\cite[Definition 1]{degen}.

\begin{definition}\label{def-finitenessclass}
A \textbf{finiteness class} is a class $\mathscr F$ of sets satisfying the following four properties:
\begin{enumerate}
\item If $X\in\mathscr F$ and $Y\subseteq X$, then $Y\in\mathscr F$,
\item If $X\in\mathscr F$ and $Y\approx X$, then $Y\in\mathscr F$,
\item If $X$ is finite then $X\in\mathscr F$,
\item $\omega\notin\mathscr F$.
\end{enumerate}
\end{definition}

So finiteness classes are classes defined by a formula which in a sense provides a notion of smallness. Due to the last two clauses in Definition~\ref{def-finitenessclass}, the class \fin  of all finite sets constitutes the smallest finiteness class and the class {\bf D-}\fin of all Dedekind finite sets constitutes the largest finiteness class. We now proceed to state various definitions of finiteness that have been considered in the literature. Each of these definitions determines a finiteness class (that of all objects that satisfy the definition); as stated below, the first three of these definitions appear in~\cite[Def. 8]{herrlich-finiteinfinite} and the last one is from~\cite[Def. 1.3]{herrlich-howard-tachtsis}, although all of these concepts have appeared before in the literature under different names~\cite{truss-classesofdedekind,howard-yorke,degen,goldstern-stronglyamorphous}.

\begin{definition}\label{def-finiteness}
Let $X$ be a set.
\begin{enumerate}
\item We say that $X$ is \textbf{A-finite} if every subset of $X$ is either finite or co-finite (equivalently, $X$ cannot be partitioned into two infinite pieces\footnote{Traditionally, sets that are infinite and A-finite are known as {\em amorphous} sets.}),
\item we say that $X$ is \textbf{B-finite} if no infinite subset of $X$ is linearly orderable,
\item we say that $X$ is \textbf{C-finite} if there is no surjection $f:X\longrightarrow\omega$, and
\item we say that $X$ is \textbf{E-finite} if there is no proper subset $Y\subsetneq X$ which surjects onto $X$.
\item We say that $X$ is \textbf{A-infinite}, \textbf{B-infinite}, \textbf{C-infinite}, or \textbf{E-infinite}, respectively, if $X$ is not A-finite, not B-finite, not C-finite, or not E-finite, respectively.
\end{enumerate}
\end{definition}
The class {\bf X-}\fin of all X-finite sets is a finiteness class, for $\mathrm{X}\in\{\mathrm{A,B,C,E}\}$. It turns out that, in $\zf$, A-finite implies both B-finite and C-finite; C-finite implies E-finite, and none of these implications is reversible; furthermore, there are no other $\zf$-provable implications between these notions, other than the ones just mentioned. The reader interested in finding references for these results, as well as for tracking other names with which these notions of finiteness have appeared previously in the literature, should consult~\cite[Section 1]{herrlich-howard-tachtsis}.

In this paper, we consider finiteness classes defined in terms of the failure of certain Ramsey-theoretic statements. The general flavour of Ramsey theory is that every time one partitions some large enough structure (usually it is said that one ``colours'' the structure, each piece of the partition representing a different colour), one can find some interesting substructures completely contained within one piece of the partition (one talks about finding \textit{monochromatic} such substructures). Sometimes ``large enough'' means infinite; hence, it seems natural to analyze whether the failure of some of these kinds of statements at a given set provides us with a valid finiteness class\footnote{Another exciting line of research consists in investigating whether certain Ramsey-theoretic statements on ordinal numbers (that have been proved in $\zfc$) must still hold in the theory $\zf$, see e.g.~\cite{philipp-philipp-thilo-choicelessramsey}.}. This paper focuses on such an analysis for two different families of Ramsey-theoretic statements.

(A particular case of) Ramsey's classical theorem~\cite{ramsey} states that, given any infinite set $X$, whenever one colours the collection $[X]^2$ of $2$-element subsets of $X$ with $2$ colours, there exists an infinite subset $Y\subseteq X$ such that all $2$-element subsets of $Y$ have the same colour; in other words, for every $c:[X]^2\longrightarrow 2$ there exists an infinite $Y\subseteq X$ such that $c\upharpoonright[Y]^2$ is a constant function\footnote{One can think of $[X]^2$ as the edge set of a complete graph with vertex set $X$. Hence, Ramsey's theorem is usually phrased by stating that whenever one colours, with two colours, the edges of an infinite complete graph, it is possible to find an infinite induced subgraph involving only one of the colours.}. As shown by Kleinberg~\cite{kleinberg-ramsey}, proving this theorem necessarily requires some form of the Axiom of Choice; Blass~\cite{blass-on-ramsey} precisely located this principle among the hierarchy of weak forms of the Axiom of Choice. Without the axiom of choice, it is possible to have infinite sets without this property. Thus, the property of satisfying the negation of Ramsey's theorem becomes yet another notion of a set being ``small''; this allows us to define a notion of finiteness based on this property. Therefore we can say that a set $X$ is \textbf{R-finite} if it does not satisfy Ramsey's theorem; in other words, if there exists a colouring $c:[X]^2\longrightarrow 2$ such that for every infinite $Y\subseteq X$, the mapping $c\upharpoonright[Y]^2$ is not constant.

Since Ramsey's theorem for $\omega$ is provable in $\zf$, it is easy to see that every R-finite set must be D-finite; also, if $X$ is finite, it has no infinite subsets, and therefore any colouring of $[X]^2$ vacuously witnesses that $X$ is R-finite. Thus, the definition of R-finite constitutes a notion of finiteness which is intermediate between finite and D-finite. In this paper we will study various notions related to R-finiteness.

Another cornerstone of infinitary Ramsey theory is the statement known as Hindman's theorem. There are two ways of stating this result, and these two statements were known to imply each other even before either of them was known to be true (see~\cite{baumgartner-short-proof-of-hindman}; the translation from one to the other is based on the fact that natural numbers are naturally identified with finite subsets of $\omega$ by considering the set of positions in which the binary expansion of a natural number features a non-zero digit); below we state these two results.

\begin{theorem}[\cite{hindmanhisthm}]\label{hindmanthm}
\hfill
\begin{enumerate}
\item\label{hindmansums} For every colouring $c:\mathbb N\longrightarrow 2$ of the natural numbers $\mathbb N$ with $2$ colours, there exists an infinite set $X\subseteq\mathbb N$ such that all elements of the set
\begin{equation*}
\fs(X)=\left\{\sum_{x\in F}x\bigg|F\in[X]^{<\omega}\setminus\{\varnothing\}\right\}
\end{equation*}
(the set of all sums of finitely many elements of $X$) have the same colour.

\item\label{hindmanunions} For every colouring $c:[\omega]^{<\omega}\longrightarrow 2$ of all finite subsets of $\omega$ with $2$ colours, there exists an infinite set $X\subseteq[\omega]^{<\omega}$, whose elements are pairwise disjoint, such that all elements of the set
\begin{equation*}
\fu(X)=\left\{\bigcup_{x\in F}x\bigg|F\in[X]^{<\omega}\setminus\{\varnothing\}\right\}
\end{equation*}
(the set of all unions of finitely many elements of $X$) have the same colour.
\end{enumerate}
\end{theorem}

The first statement in Theorem~\ref{hindmanthm} involves the addition operation in $\mathbb N$, so attempting to generalize it to other sets requires ensuring that said sets can be endowed with a suitable semigroup operation. Working in $\zf$, one can trivially endow every set $X$ with {\em some} semigroup operation, e.g. defining $x*y=y$ for every $x,y\in X$; however, this particular semigroup operation trivializes Hindman's theorem: with this semigroup operation we would have, given $Y\subseteq X$, that any product of finitely many elements of $Y$ is itself an element of $Y$ (the rightmost one in the product) and so the analog of $\fs(Y)$ is simply $Y$---thus making the analog of Hindman's theorem equivalent to the pigeonhole principle. If one attempts, however, to equip a set $X$ with a semigroup operation with nicer properties, similar to the ones enjoyed by the addition on $\mathbb N$, one might run into difficulties if working in $\zf$: it is a classical theorem of Hajnal and Kert\'esz~\cite{hajnal-kertesz} that the statement ``every set can be equipped with a cancellative commutative semigroup operation'', as well as the statement ``every set can be equipped with a cancellative semigroup operation'', are equivalent to the Axiom of Choice (it is also well-known\footnote{The interested reader is referred to~\cite[Chapter I.6]{rubin-rubin} for more statements, referring to the possibility of endowing every set with certain algebraic structure, that are equivalent to $\ac$.} that $\ac$ is equivalent to the statement ``every set can be equipped with a group operation''). Hence, if one wants to work in $\zf$, it seems more fruitful to focus on the second statement from Theorem~\ref{hindmanthm}, since this statement does not make any reference to a semigroup operation, and analogs of this statement for any set $X$ can be readily formulated---in fact, the class of all sets satisfying such a statement is easily seen to form a finiteness class. Thus we will say that a set $X$ is \textbf{H-infinite} if for every colouring $c:[X]^{<\omega}\longrightarrow 2$ there exists an infinite, pairwise disjoint family $Y\subseteq[X]^{<\omega}$ such that the set $\fu(Y)$ is monochromatic, and \textbf{H-finite} otherwise. In this paper we will also study various notions of finiteness that are closely related to the one just defined.

Thus, the main objective of this paper is to analyze in detail various finiteness classes, all of them closely related either to the class of R-finite sets, or to the class of H-finite sets, as defined above. We managed to get fairly complete information regarding which of these classes are (provably in $\zf$) contained in one another. The second section of this paper deals with the various finiteness classes that arise from versions of Ramsey's theorem, whereas the third section deals with those that arise from various ways of stating Hindman's theorem; in both sections we focus on implications that are provable in $\zf$. Finally, the fourth section delves deep into a study of various Fraenkel--Mostowski permutation models, which allows us to establish what implications between the notions of finiteness considered here are not provable in $\zf$. There is a short fifth section mentioning a few questions that remain open.

\section{Flavours of Ramsey finiteness}

In the introduction we mentioned Ramsey's theorem for pairs of elements. A more general version of Ramsey's theorem (provable in $\zfc$) states that, given any infinite set $X$, and any finite number $n<\omega$, for every colouring $c:[X]^n\longrightarrow 2$ of the collection $[X]^n$ of $n$-element subsets of $X$ with two colours, one can find an infinite $Y\subseteq X$ such that $c\upharpoonright[Y]^n$ is a constant function. Hence, one can define infinitely many finiteness classes arising from Ramsey's theorem, one for each natural number $n$; we do so in the definition below.

\begin{definition}
Let $X$ be a set and let $n\in\mathbb N$.
\begin{enumerate}
\item We say that $X$ is \textbf{\r-finite} if it does not satisfy Ramsey's theorem for $n$-element subsets; in other words, $X$ is \r-finite if there exists a colouring $c:[X]^n\longrightarrow 2$ such that for every infinite $Y\subseteq X$, the mapping $c\upharpoonright[Y]^n$ is not constant.
\item We say that $X$ is \textbf{\r-infinite} if it is not \r-finite.
\end{enumerate}
\end{definition}

We start with some fairly easy observations. First of all note that, by the pigeonhole principle (or rather, by the statement, provable in $\zf$, that a finite union of finite sets must be finite), every infinite set is \ra{1}-infinite. In particular, the class \ra{1}-\fin of all \ra{1}-finite sets is a finiteness class (that coincides with the class of all finite sets); for $n\in\mathbb N\setminus\{1\}$, it is easy to see that the class \r-\fin of all \r-finite sets also constitutes a finiteness class, but in this case, as we will show, the corresponding finiteness class is (consistently) properly between the class of all finite sets and the class of all D-finite sets.

Surprisingly, we have not been able to find any implication between the \r-finiteness and the \ra{m}-finiteness of a given set, for $n\neq m$, unless one assumes certain additional structure on that set. The usual tricks for deriving Ramsey's theorem for $n$-tuples from Ramsey's theorem for $m$-tuples, which work when the underlying set is $\omega$, seem to rely on the linear orderliness of $\omega$ (in the case when $n<m$) and possibly also its well-orderliness (in the case where $n=m+1$ and one proceeds by induction). As we see below, assuming at least a linear order in our set allows us to carry out one of these arguments.

\begin{proposition}\label{prop:rn-to-rn+1}
Suppose that $X$ is a linearly orderable set. If $X$ is \r-finite then $X$ is \ra{n+1}-finite (and consequently $X$ will be \ra{m}-finite for all $m>n$).
\end{proposition}

\begin{proof}
We proceed by contrapositive, so let us assume that $X$ is \ra{n+1}-infinite, and fix a linear order $\leq$ on $X$. Now suppose that we have a colouring $c:[X]^n\longrightarrow 2$. Define a colouring $d:[X]^{n+1}\longrightarrow 2$ by $d(x)=c(x\setminus\{\min(x)\})$, where the minimum is taken with respect to the linear ordering $\leq$ (such a minimum exists because $x$ is a finite set). By assumption, we obtain an infinite $Y\subseteq X$ such that $[Y]^{n+1}$ is monochromatic for $d$, and let $i$ be the corresponding colour. Define
\begin{equation*}
Y'=\begin{cases}
Y\setminus\{\min(Y)\};\text{ if }Y\text{ has a minimum}, \\
Y;\text{ otherwise}.
\end{cases}
\end{equation*}
We claim that $[Y']^n$ is monochromatic for $c$ on colour $i$. To see this, take an arbitrary $y\in[Y']^n$. Let $y'$ be an element of $Y$ which is smaller than every element of $y$ according to the linear order $\leq$ (our definition of $Y'$ ensures that we can always find such a $y'$). Then $\{y'\}\cup y\in[Y]^{n+1}$ and therefore $i=d(\{y'\}\cup y)=c((\{y'\}\cup y)\setminus\{\min(\{y'\}\cup y)\})=c((\{y'\}\cup y)\setminus\{y'\})=c(y)$, and we are done. Hence, $[Y']^n$ is monochromatic (on colour $i$) and so $X$ is \r-infinite, which finishes the proof.
\end{proof}

We will prove in Section~\ref{Sect:Independence} that the converse of Proposition~\ref{prop:rn-to-rn+1} does not hold. Namely, Corollary~\ref{bpi-model} shows that, in general, \ra{m}-finite does not imply \r-finite when $n<m$, even for linearly ordered sets (in fact, not even if one assumes the Boolean Prime Ideal theorem). The following proposition uses a well-known idea---utilizing a partition of a set to colour its pairs of elements---to relate \r-finiteness and the possibility of choosing elements from a partition into finite sets.

\begin{proposition}\label{russellsetrfinite}
Let $n\in\mathbb N\setminus\{1\}$, and let $X$ be an \r-infinite set. If $\mathscr F$ is any infinite partition of $X$ into finite pieces, then there exists an infinite subfamily of $\mathscr F$ that admits a choice function.
\end{proposition}

\begin{proof}
Given the partition $\mathscr F$ of the set $X$, define a colouring $c:[X]^n\longrightarrow 2$ of the $n$-element subsets of $X$ by
\begin{equation*}
c(x)=\begin{cases}1;\text{ if }(\exists F\in\mathcal F)(|x\cap F|\geq 2); \\
0;\text{ if }(\forall F\in\mathcal F)(|x\cap F|\leq 1).\end{cases}
\end{equation*}
Since $X$ is \r-infinite, let $Y\subseteq X$ be infinite such that $[Y]^n$ is monochromatic in some colour $i\in 2$. The assumption that each $F\in\mathcal F$ is finite implies that the subfamily $\mathcal F'=\{F\in\mathcal F\big|Y\cap F\neq\varnothing\}$ is infinite. Choosing distinct $F_1,\ldots,F_n\in\mathcal F'$ and choosing elements $y_i\in F_i$, one gets an $n$-tuple $y=\{y_1,\ldots,y_n\}\in[Y]^n$ with $c(y)=0$; therefore we have that $i=0$. This implies that $|Y\cap F|=1$ for each $F\in\mathcal F'$, hence $Y$ is a selector for $\mathcal F'$.
\end{proof}

As a particular case of Proposition~\ref{russellsetrfinite}, we obtain that Russell sets\footnote{A {\bf Russell set} is a set that can be partitioned into countably many cells of cardinality $2$, in such a way that no infinite subfamily of the partition admits a choice function. The terminology arises from B. Russell's observation that a family of infinitely many pairs of socks has no choice function (assuming that both socks within a pair are always indistinguishable; this assumption is quite controversial).} are \r-finite for every $n\geq 2$. Now, recall that a set is said to be {\em amorphous} if it is both infinite and A-finite. The rest of the section will be devoted to show that the converse of Proposition~\ref{russellsetrfinite} holds for amorphous sets. This will allow us to thoroughly analyze the relation between amorphous sets and \r-finite or \r-infinite sets, for various $n\in\mathbb N\setminus\{1\}$. We begin with a lemma that deals with amorphous graphs.

\begin{lemma}\label{componentsgraph}
Let $G=(V,E)$ be a locally finite graph, with $V$ an amorphous set. Then, each connected component of $G$ is finite.
\end{lemma}

\begin{proof}
Let $X$ be a connected component of $G$. Upon choosing an $x\in X$, we can inductively define sets $N_x^n$, for $n<\omega$, by
\begin{eqnarray*}
N_x^0 & = & \{x\}, \\
N_x^{n+1} & = & \left\{y\in V\setminus\left(\bigcup_{i=0}^n N_x^i\right)\bigg|(\exists z\in N_x^n)(\{y,z\}\in E)\right\}.
\end{eqnarray*}
Notice that the set $N_x^n$ consists of the vertices that are at distance exactly $n$ from $x$, and that $X=\bigcup_{n<\omega}N_x^n$. Thus, we have a function $f:X\longrightarrow\omega$ by letting $f(y)$ be the unique $n$ such that $y\in N_x^n$. Since $X$ is A-finite, the range of $f$ must be finite (otherwise we would be able to partition said range in two infinite co-infinite sets, which would then induce a partition of $X$ in two infinite co-infinite sets via $f$-preimages, contradicting A-finiteness of $X$). Furthermore, since the graph $G$ is locally finite, it is straightforward to prove by induction on $n$ that each of the sets $N_x^n$ is finite. Therefore, $X=\bigcup_{n<\omega}f^{-1}[\{n\}]=\bigcup_{n\in\ran(f)}N_x^n$ is a finite union of finite sets, and is thus a finite set itself.
\end{proof}

There is a very close relation between colourings of pairs of a set and graphs defined on that set. After the following definition, we will explore this in the context of amorphous sets.

\begin{definition}\label{density}
Let $X$ be a set, let $n\in\mathbb N\setminus\{1\}$, and let $c:[X]^n\longrightarrow 2$ be a colouring.
\begin{enumerate}
\item We will say that $c$ is {\bf dense} if there exists an infinite set $Y\subseteq X$ such that for all $\{x_1,\ldots,x_{n-1}\}\in[Y]^{n-1}$ there are $y,z\in X\setminus\{x_1,\ldots,x_{n-1}\}$ such that $c(\{x_1,\ldots,x_{n-1},y\})\neq c(\{x_1,\ldots,x_{n-1},z\})$.
\item If $n=2$ and $i<2$, we will say that $c$ is {\bf $i$-locally finite} if the graph with vertex set $X$ and edge set $\left\{e\in[X]^2\big|c(e)=i\right\}$ is locally finite.
\end{enumerate}
\end{definition}

\begin{lemma}\label{colouringgraph}
Let $X$ be an amorphous set, and let $c:[X]^2\longrightarrow 2$. Then, there exists a unique $i<2$ and a cofinite set $Y\subseteq X$ such that $c\upharpoonright[Y]^2$ is $i$-locally finite.
\end{lemma}

\begin{proof}
For each $x\in X$, there is a partition $X=\{x\}\cup F_0^x\cup F_1^x$, where
\begin{eqnarray*}
F_0^x & = & \{ y\in X\setminus\{x\}\big|c(x,y)=0\}, \\
F_1^x & = & \{ y\in X\setminus\{x\}\big|c(x,y)=1\}.
\end{eqnarray*}
Since $X$ is amorphous, exactly one of $F_0^x$ and $F_1^x$ is finite. This induces a partition $X=F_0\cup F_1$, where
\begin{eqnarray*}
F_0 & = & \{x\in X\big|F_0^x\text{ is finite}\}, \\
F_1 & = & \{x\in X\big|F_1^x\text{ is finite}\};
\end{eqnarray*}
once again, by amorphousness of $X$, there exists an $i<2$ such that $Y=F_i$ is cofinite. Then the graph $G=(Y,E)$ with $E=\left\{e\in[Y]^2\big|c(e)=i\right\}$ is locally finite, for if $x\in Y$, the set of neighbours of $x$ in $G$ is precisely $Y\cap F_i^x$, which is a finite set by assumption. The uniqueness of $i$ follows from the fact that $F_{1-i}^x$ is infinite for cofinitely many $x\in X$.
\end{proof}

We now start analyzing the other notion introduced in Definition~\ref{density}, namely that of a dense colouring. Notice that, if $X$ is any set and $c:[X]^n\longrightarrow 2$ satisfies that there is an infinite $Y\subseteq X$ with $[Y]^n$ monochromatic, then, in particular, $c\upharpoonright[Y]^n$ fails to be a dense colouring. The following proposition illustrates a situation in which the converse of this statement also holds.

\begin{proposition}\label{densevsgood}
Let $X$ be any infinite set, $n\in\mathbb N$, and $c:[X]^n\longrightarrow 2$ a colouring. Suppose that all infinite subsets of $X$ are \ra{n-1}-infinite, and that there is an infinite $Y\subseteq X$ such that $c\upharpoonright[Y]^n$ fails to be dense. Then there exists an infinite $Z\subseteq X$ such that $[Z]^n$ is monochromatic for $c$.
\end{proposition}

\begin{proof}
That $c\upharpoonright[Y]^n$ fails to be dense means that for every infinite $Y'\subseteq Y$, there is $\{x_1,\ldots,x_{n-1}\}\in[Y']^{n-1}$ and an $i<2$ such that for every $x\in Y\setminus\{x_1,\ldots,x_{n-1}\}$, $c(\{x_1,\ldots,x_{n-1},x\})=i$. So, if we let $d:[Y]^{n-1}\longrightarrow 2$ be defined by $d(\{x_1,\ldots,x_{n-1}\})=0$ iff $(\exists i<2)(\forall x\in Y\setminus\{x_1,\ldots,x_{n-1}\})(c(\{x_1,\ldots,x_{n-1},x\})=i)$, and use the fact that $Y$ is \ra{n-1}-infinite to get an infinite $Y'\subseteq Y$ such that $[Y']^{n-1}$ is monochromatic for $d$, then this set must in fact be monochromatic in colour $0$. Hence, we can define $e:[Y']^{n-1}\longrightarrow 2$ by letting $e(\{x_1,\ldots,x_{n-1}\})$ be the unique $i<2$ such that $(\forall x\in Y\setminus\{x_1,\ldots,x_{n-1}\})(c(\{x_1,\ldots,x_{n-1},x\})=i)$. Using now the fact that $Y'$ is \ra{n-1}-infinite, obtain an infinite $Z\subseteq Y'$ such that $[Z]^{n-1}$ is monochromatic for $e$. It is then readily checked that $[Z]^n$ is in fact monochromatic for $c$.
\end{proof}

Since every infinite set is \ra{1}-infinite, the previous proposition applies to every infinite set when $n=2$. The following corollary will be of crucial importance in the remainder of this section.

\begin{corollary}\label{denseiffbad}
Let $n\in\mathbb N$, let $X$ be an amorphous, \ra{n-1}-infinite set, and let $c:[X]^n\longrightarrow 2$ be a colouring. Then, there exists an infinite $Y\subseteq X$ with $[Y]^n$ monochromatic if and only if there exists an infinite $Y\subseteq X$ such that the colouring $c\upharpoonright[Y]^n$ fails to be dense. In particular, if $X$ is amorphous and \ra{2}-finite, then a colouring $c:[X]^2\longrightarrow 2$ witnesses this fact if and only if for every infinite $Y\subseteq X$, the colouring $c\upharpoonright[Y]^2$ is dense.
\end{corollary}

\begin{proof}
The forward direction holds for any set; for the backward direction, note that the property of being \ra{n-1}-infinite is invariant under finite changes, and so if $X$ is amorphous and \ra{n-1}-infinite, then so are all of its infinite subsets. Thus, we can apply Proposition~\ref{densevsgood}.

The second statement follows immediately from the fact that every infinite set is \ra{1}-infinite.
\end{proof}

\begin{lemma}\label{denselocallyfinite}
Let $X$ be an amorphous set, and let $c:[X]^2\longrightarrow 2$ be a dense colouring which is $i$-locally finite for some $i<2$. Then, there is a partition of $X$ into finite pieces, no infinite subfamily of which admits a choice function. In particular, by Proposition~\ref{russellsetrfinite}, $X$ is \r-finite for all $n\in\mathbb N\setminus\{1\}$.
\end{lemma}

\begin{proof}
Let $\mathcal F$ be the family of connected components of the (locally finite) graph with vertex set $X$ and edge set $\{e\in[X]^2\big|c(e)=i\}$. By Lemma~\ref{componentsgraph}, each element  of $\mathcal F$ is finite. Furthermore, since $c$ is a dense colouring, all but finitely many elements of the family $\mathcal F$ contain at least two elements. Hence, if $\mathcal F'$ was an infinite subfamily of $\mathcal F$ with a selector $Y$, then $Y$ would be infinite, and so would $X\setminus Y$ (since for infinitely many $F\in\mathcal F'$, there would be at least one element in $F\setminus Y$), contradicting that $X$ is amorphous. Therefore, no infinite subfamily of $\mathcal F$ can possibly admit a choice function, and we are done.
\end{proof}

We are finally able to prove the main theorem for this section, which shows that the distinct notions of \r-finiteness, as $n$ varies, cannot be separated within the class of amorphous sets. The proof of this theorem shows that, if $X$ is amorphous and \r-finite for some $n\in\mathbb N\setminus\{1\}$, then $X$ carries a partition into finite sets, no infinite subfamily of which admits a choice function. This last statement constitutes the converse of Proposition~\ref{russellsetrfinite} for amorphous sets, as announced earlier.

\begin{theorem}\label{amorphousrfinite}
Let $X$ be an amorphous set. Then either $X$ is \r-finite for all $n\in\mathbb N\setminus\{1\}$ or $X$ is \r-infinite for all $n\in\mathbb N$.
\end{theorem}

\begin{proof}
Let $X$ be an amorphous set. The proof breaks into two cases:
\begin{description}
\item[Case 1] If $X$ is \ra{2}-finite, then by Corollary~\ref{denseiffbad}, there is a colouring $c:[X]^2\longrightarrow 2$ such that for every infinite $Y\subseteq X$, $c\upharpoonright[Y]^2$ is dense. By Lemma~\ref{colouringgraph}, we can pick an $i<2$ and a cofinite $Y\subseteq X$ such that $c\upharpoonright[Y]^2$ is $i$-locally finite (and still dense, by assumption). Hence, $Y$ (and therefore also $X$) is \r-finite for all $n\in\mathbb N\setminus\{1\}$, by Lemma~\ref{denselocallyfinite}.

\item[Case 2] If $X$ is \ra{2}-infinite, we will proceed to argue that $X$ must in fact be \r-infinite for all $n\in\mathbb N$. Suppose otherwise, let $n$ be the least number such that $X$ is \r-finite (note that we must have $n\geq 3$), and let $c:[X]^n\longrightarrow 2$ be a witness for this. 
Now, for each $F\in[X]^{n-1}$, there is an $i_F<2$ such that all but finitely many $x\in X\setminus F$ satisfy $c(F\cup\{x\})=i_F$. We let $d:[X]^{n-1}\longrightarrow 2$ be given by $d(F)=i_F$, and use the fact that $X$ is \ra{n-1}-infinite to get a cofinite $Y\subseteq X$ and an $i<2$ such that $[Y]^{n-1}$ is monochromatic for $d$ in colour $i$. This means that for all $\{x_1,\ldots,x_{n-1}\}\in[Y]^{n-1}$, there are only finitely many $x\in Y\setminus\{x_1,\ldots,x_{n-1}\}$ with $c(\{x_1,\ldots,x_{n-1},x\})=1-i$. Now, since $c$ witnesses that $X$ is \r-finite, by Corollary~\ref{denseiffbad}, the colouring $c\upharpoonright[Y]^n$ is dense, meaning that there is a cofinite $Z\subseteq Y$ such that for all $\{x_1,\ldots,x_{n-1}\}\in[Z]^{n-1}$ there are $x,y\in Y\setminus\{x_1,\ldots,x_{n-1}\}$ such that $c(\{x_1,\ldots,x_{n-1},x\})\neq c(\{x_1,\ldots,x_{n-1},y\})$.
Pick an arbitrary $n-2$-sized set $\{x_1,\ldots,x_{n-2}\}\in[Z]^{n-2}$, and define the colouring $e:[Y\setminus\{x_1,\ldots,x_{n-2}\}]^2\longrightarrow 2$ by $e(\{x,y\})=c(\{x_1,\ldots,x_{n-2},x,y\})$. Then we have, on the one hand, that for each $x\in Y\setminus\{x_1,\ldots,x_{n-2}\}$ there are only finitely many $y\in Y\setminus\{x_1,\ldots,x_{n-2},x\}$ such that $e(\{x,y\})=c(\{x_1,\ldots,x_{n-2},x,y\})=1-i$, and so $e$ is $1-i$-locally finite; and on the other hand, for each $x$ that belongs to the infinite subset $Z\setminus\{x_1,\ldots,x_{n-2}\}$ of $Y\setminus\{x_1,\ldots,x_{n-2}\}$, there are $y,z\in Y\setminus\{x_1,\ldots,x_{n-2},x\}$ such that $e(\{x,y\})=c(\{x_1,\ldots,x_{n-2},x,y\})\neq c(\{x_1,\ldots,x_{n-2},x,z\})=e(\{x,y\})$, and so $e$ is a dense colouring. Hence, by Lemma~\ref{denselocallyfinite}, the set $Y\setminus\{x_1,\ldots,x_{n-2}\}$ (and hence also $X$) must be \r-finite for all $n\in\mathbb N\setminus\{1\}$, a contradiction.
\end{description}
\end{proof}

If one wants to delve deep into the classification of amorphous sets, there is a neat relation between this and the different notions of \r-finiteness (as $n$ varies). Recall that Truss~\cite{truss-structureamorphous} defines a {\em strongly amorphous} set as an amorphous set which can only be partitioned into finite pieces if all but finitely many of those pieces are singletons.

\begin{theorem}
If $X$ is a strongly amorphous set, then $X$ is \r-infinite for some $n\in\mathbb N\setminus\{1\}$ (equivalently, for all $n\in\mathbb N\setminus\{1\}$).
\end{theorem}

\begin{proof}
Suppose that $X$ is strongly amorphous, and let $c:[X]^2\longrightarrow 2$ be any colouring. Using Lemma~\ref{colouringgraph}, find a colour $i<2$ and a cofinite subset $Z\subseteq X$ such that the graph $(Z,E)$, with $E=\{e\in[Z]^2\big|c(e)=i\}$, is locally finite. Then if we let $\mathcal F$ be the partition of $Y$ into the connected components of $G$, each of the pieces of this partition is finite by Lemma~\ref{componentsgraph}. Thus, $\mathcal F\cup\{X\setminus Z\}$ is a partition of $X$ into finite pieces; since $X$ is strongly amorphous, all but finitely many of the elements of this partition are singletons. Let $\mathcal F'=\{F\in\mathcal F\big||F|=1\}$ and let $Y=\bigcup_{F\in\mathcal F'}F$. Then $Y\subseteq X$ is infinite and any two distinct $x,y\in Y$ lie in distinct connected components of $G$, in other words, $c(\{x,y\})=1-i$. Hence, $[Y]^2$ is monochromatic, and we are done.
\end{proof}

\section{Flavours of Hindman finiteness}

In the introduction we mentioned two equivalent statements, each of which can be called ``Hindman's theorem''. The first one is that, for every colouring $c:\mathbb N\longrightarrow 2$ of the natural numbers $\mathbb N$ with two colours, there exists an infinite set $X\subseteq\mathbb N$ such that the set $\fs(X)$ is monochromatic; and the second is that for every colouring $c:[\omega]^{<\omega}\longrightarrow 2$ of all finite subsets of $\omega$ with two colours, there exists an infinite pairwise disjoint family $X\subseteq[\omega]^{<\omega}$ such that the set $\fu(X)$ is monochromatic. Here, $\fs(X)=\left\{\sum_{x\in F}x\bigg|F\in[X]^{<\omega}\setminus\{\varnothing\}\right\}$ is the set of all sums of finitely many elements from $X$, whereas $\fu(X)=\left\{\bigcup_{x\in F}x\bigg|F\in[X]^{<\omega}\setminus\{\varnothing\}\right\}$ is the set of all unions of finitely many elements of $X$.

Each of these two statements has generalizations that go beyond the scope of Hindman's original theorem. The first statement (the one in terms of finite sums) generalizes, by using the tools of algebra in the \v{C}ech--Stone compactification~\cite{hindman-strauss}, to the statement that for every infinite abelian group $G$, and for every $c:G\longrightarrow 2$, there exists an infinite $X\subseteq G$ such that the set $\fs(X)$ (defined exactly as above, but with respect to the group operation of $G$) is monochromatic\footnote{The group $G$ need not be abelian, but the non-commutative case requires a \textit{sequence}, rather than a set, of elements of $G$, to ensure that all elements are multiplied in the same order when computing finite products.}. Meanwhile, the second statement (the one in terms of finite unions) readily generalizes, under $\ac$, to the statement that for every infinite set $X$ and for every $c:[X]^{<\omega}\longrightarrow 2$ there exists an infinite set $Y\subseteq[X]^{<\omega}$, whose elements are pairwise disjoint, such that the set $\fu(Y)$ is monochromatic. The second generalization immediately lends itself to formulating yet another definition of finiteness, whereas the first one not so much, since without $\ac$ there could be sets that cannot be endowed with any group operation. It is possible, however, to focus on certain specific group structures, most notably that of the {\it Boolean group} on a set: given a set $X$, its finite powerset $[X]^{<\omega}$ forms a group when equipped with the symmetric difference $\bigtriangleup$ as group operation; this group is called Boolean because each of its non-identity elements has order $2$. Note also that, if $Y\subseteq[X]^{<\omega}$ is a family of pairwise disjoint elements, then we actually have that $\fs(Y)=\fu(Y)$, where the left-hand side is interpreted as the set of finite sums computed in the Boolean group $[X]^{<\omega}$. It turns out that, for many applications of Hindman's theorem, considering only Boolean groups is general enough (for example,~\cite[Corollary 3.2]{strongly-union-trivialsums} implies that, if one is interested in strongly summable ultrafilters on abelian groups, without loss of generality one can assume that the relevant group is Boolean).

Hence, one immediately sees at least two possible ways of defining a finiteness class inspired by Hindman's theorem: in order for a set $X$ to be large, one could require that for every colouring $c:[X]^{<\omega}\longrightarrow 2$ there exists an infinite $Y\subseteq[X]^{<\omega}$ such that $\fs(Y)$ (computed in the Boolean sense) is monochromatic, or one could require the same condition but with the additional requirement that the family $Y$ is pairwise disjoint. But there is, in fact, much more variability here, for one could define (conceivably) weaker versions of largeness by, instead of asking that the full set $\fs(Y)$ be monochromatic, requiring only that the more restricted set $\fs_{\leq k}(Y)=\{\sum_{y\in F}y\big|F\subseteq Y\text{ and }0<|F|\leq k\}$ be monochromatic. Considerations about Hindman's theorem for a restricted number of summands have been pondered in the literature: it is a very old question of Hindman, Leader and Strauss~\cite[Question 12]{hindman-leader-strauss-partition-regularity} whether a proof of Hindman's theorem for at most two summands already implies the full Hindman's theorem for any finite number of summands. Although this question was asked in a vague sense, there are ways of making this question precise, for example, in terms of computability theory~\cite[Question 11]{blass-questions-hindman}. From the perspective of computability theory, various authors have made progress on the question whether Hindman's theorem for a bounded number of summands is as strong as the full version of Hindman's theorem~\cite{efectiveness-hindman,carlucci-et-al-bounds-on-hindman} (see also~\cite{6-authors-reverse-math-hindman} for results on the strength of Hindman's theorem when restricted to {\em exactly} (rather than {\em at most}) two summands). For another example of these kinds of considerations, Carlucci~\cite{carlucci-hindman} has some results about colouring uncountable groups (with finitely many colours) and obtaining monochromatic uncountable sets (of some prescribed cardinality) of the form $\fs_a(X)=\{\sum_{x\in F}x\big|F\in[X]^k\text{ for some }k\in a\}$, for a variety of different kinds of finite sets $a$ (e.g. if $a$ is an arithmetic progression, or even a finite set of the form $\fs(Z)$).

Here we explore a different approach to the issue of making this question precise, from the perspective of choiceless set theory; the answer in this case is rather surprising. The discussion above gives rise to the following infinite family of definitions.

\begin{definition}
Let $X$ be a set, let $n<\omega$, and let $Y\subseteq[X]^{<\omega}$ be a collection of elements of the Boolean group based on $X$.
\begin{enumerate}
\item If the set $Y$ consists of pairwise disjoint elements, we will occasionally write $\fu_{\leq n}(Y)$ instead of $\fs_{\leq n}(Y)$, given that for pairwise disjoint elements, the sum is the same as the union.
\item The set $X$ will be said to be \hb{n}-finite if there exists a colouring $c:[X]^{<\omega}\longrightarrow 2$ such that for every infinite $Y\subseteq[X]^{<\omega}$, the set $\fs_{\leq n}(Y)$ is not monochromatic.
\item The set $X$ will be said to be \hb{}-finite if there exists a colouring $c:[X]^{<\omega}\longrightarrow 2$ such that for every infinite $Y\subseteq[X]^{<\omega}$, the set $\fs(Y)$ is not monochromatic.
\item The set $X$ will be said to be \hd{n}-finite if there exists a colouring $c:[X]^{<\omega}\longrightarrow 2$ such that for every infinite $Y\subseteq[X]^{<\omega}$ consisting of pairwise disjoint elements, the set $\fu_{\leq n}(Y)$ is not monochromatic.
\item The set $X$ will be said to be \hd{}-finite if there exists a colouring $c:[X]^{<\omega}\longrightarrow 2$ such that for every infinite $Y\subseteq[X]^{<\omega}$ consisting of pairwise disjoint elements, the set $\fu(Y)$ is not monochromatic.
\end{enumerate}
\end{definition}

Hence, a set $X$ is \hd{n}-infinite if it satisfies the pairwise disjoint version of Hindman's theorem for up to $n$ summands, and \hb{n}-infinite if it satisfies the Boolean group version of this theorem, again for up to $n$ summands; removing the subindices in each of these notions amounts to stating that the corresponding {\it full} version of Hindman's theorem, without restrictions on the number of summands, is satisfied. From this, and keeping in mind that the requirement of pairwise disjointness turns the pairwise disjoint version of Hindman's theorem into a stronger statement than the corresponding Boolean version, we immediately get the following implications for a set $X$ (note that, trivially, any set $X$ is finite if and only if it is \hb{1}-finite, which happens if and only if it is \hd{1}-finite; since given a colouring of $[X]^{<\omega}$ it suffices to restrict such colouring to the set $\{\{x\}\big|x\in X\}$ and then apply the pigeonhole principle):

\centerline{\xymatrix{
\text{Finite} \ar@{=>}[r] & \mathrm{H}^{\mathrm{B}}_2\text{-finite} \ar@{=>}[r] \ar@{=>}[d] & \mathrm{H}^{­\mathrm{B}}_3\text{-finite} \ar@{=>}[r] \ar@{=>}[d] & \mathrm{H}^{\mathrm{B}}_4\text{-finite} \ar@{=>}[r] \ar@{=>}[d] & \cdots \ar@{=>}[r] & \mathrm{H}^{\mathrm{B}}\text{-finite} \ar@{=>}[d] & \\
 & \mathrm{H}^{\mathrm{D}}_2\text{-finite} \ar@{=>}[r] & \mathrm{H}^{\mathrm{D}}_3\text{-finite} \ar@{=>}[r] & \mathrm{H}^{\mathrm{D}}_4\text{-finite} \ar@{=>}[r] & \cdots \ar@{=>}[r] & \mathrm{H}^{\mathrm{D}}\text{-finite} \ar@{=>}[r] & \text{D-finite}.
}}
Each of the above properties defines a finiteness class; each of these classes will also be denoted by {\bf X-}\fin, where X is either \hd{n} or \hb{n} for some $n<\omega$, or the same but without subindex. In spite of the apparent abundance of notions in the above diagram, we note the very surprising and notable fact that all but three of these notions can be proved to be equivalent in $\zf$.

\begin{theorem}\label{equivalencefour}
Let $X$ be an arbitrary set. Then, the following four statements are equivalent:
\begin{enumerate}
\item $X$ is \hd{}-finite.
\item $[X]^{<\omega}$ is D-finite.
\item $X$ is \hd{2}-finite.
\item $X$ is \hb{4}-finite.
\end{enumerate}
\end{theorem}

\begin{proof}
We will prove this equivalence by showing first that (1) implies (2), (2) implies (3) and (3) implies (1), thereby establishing the equivalence between statements (1), (2) and (3). Once this is done, it suffices to show that (4) implies any of (1), (2) or (3); and that any of (1), (2) or (3) implies (4); hence, we will show that (4) implies (1) and that (2) implies (4) to finish the proof.

\begin{description}
\item[(1)$\Rightarrow$(2)]
Suppose that $[X]^{<\omega}$ is D-infinite, that is, suppose that there exists an injective sequence $\langle x_n\big|n<\omega\rangle$ of elements of $[X]^{<\omega}$. Using this sequence, we can recursively build a new sequence $\langle y_n\big|n<\omega\rangle$ of nonempty, {\it pairwise disjoint} elements of $[X]^{<\omega}$. This is done as follows: make $y_0=x_{k_0}$, where $k_0$ is the least integer such that $x_{k_0}\neq\varnothing$; then, knowing $y_0,\ldots,y_n$, let $k_{n+1}$ be the least integer such that $x_{k_{n+1}}\not\subseteq\bigcup_{i=0}^n y_i$ (such an integer exists because there are only finitely many subsets of $\bigcup_{i=0}^n y_i$ and the sequence of $x_n$ is injective) and then let $y_{n+1}=x_{k_{n+1}}\setminus\bigcup_{i=0}^n y_i$.

Now suppose that we are given  $c:[X]^{<\omega}\longrightarrow 2$. We define a colouring $d:[\omega]^{<\omega}\longrightarrow 2$ by $d(x)=c\left(\bigcup_{i\in x}y_i\right)$. By Hindman's theorem on $[\omega]^{<\omega}$ (which is provable in $\zf$), $\omega$ is \hd{}-infinite, therefore we can find an infinite $Z\subseteq[\omega]^{<\omega}$ such that $\fu(Z)$ is monochromatic for $d$; let $i$ be the corresponding colour. Then we let $Y=\left\{\bigcup_{i\in x}y_i\big|x\in Z\right\}$, which is an infinite subset of $[X]^{<\omega}$, and we claim that $\fu(Y)$ is monochromatic for $c$ in colour $k$. This is because any element $y\in\fu(Y)$ is of the form
\begin{equation*}
y=\bigcup_{x\in F}\left(\bigcup_{i\in x}y_i\right)=\bigcup_{i\in\bigcup_{x\in F}x} y_i
\end{equation*}
for some finite $F\subseteq Z$; since $\bigcup_{x\in F}x\in\fu(X)$, then $c(y)=d\left(\bigcup_{x\in F}x\right)=i$, and so $\fu(Y)$ is indeed monochromatic in colour $i$. This shows that $X$ is \hd{}-infinite.

\item[(2)$\Rightarrow$(3)] 
Suppose that $X$ is \hd{2}-infinite, and we will show that $[X]^{<\omega}$ must be D-infinite. To see this, consider the colouring $c:[X]^{<\omega}\longrightarrow 2$ given by $c(x)=\lfloor{\log_2|x|}\rfloor\mod 2$. By assumption there is an infinite, pairwise disjoint $Y\subseteq[X]^{<\omega}$ such that $\fu_{\leq 2}(Y)$ is monochromatic for $c$. We will argue that every two distinct elements from $Y$ must have different cardinalities; to do this, we use the following lemma (whose content is fairly trivial, but which we state explicitly because it will be used again later).

\begin{lemma}\label{colourlog2}
Let $x$ and $y$ be two disjoint finite sets with $|x|=|y|$. Then $\lfloor\log_2(|x\cup y|)\rfloor=1+\lfloor\log_2|x|\rfloor$. In particular, the set $\{x,y,x\cup y\}$ cannot be monochromatic for $c$ if $c(z)=\lfloor\log_2|z|\rfloor\mod 2$ for all $z\in\dom(c)$.
\end{lemma}

\begin{proof}
Let $x,y$ be as in the hypothesis of the lemma. Since $x$ and $y$ are disjoint, we have that $|x\cup y|=|x|+|y|=2|x|$ and therefore
\begin{equation*}
\lfloor{\log_2 |x\cup y|}\rfloor=\lfloor{\log_2(2|x|)}\rfloor=1+\lfloor{\log_2|x|}\rfloor.
\end{equation*}
\end{proof}

Hence, if we had two distinct $x,y\in Y$ with the same cardinality, Lemma~\ref{colourlog2} would imply that $x$ and $y$ must have a colour different from that of $x\cup y$, which contradicts the assumption on $Y$ because $x,y,x\cup y\in\fu_{\leq 2}(Y)$.

We have thus established that the cardinality mapping $|\cdot|:Y\longrightarrow\omega$ is injective; this implies that, if we let $M$ be the range of that mapping, then $|\cdot|:Y\longrightarrow M$ is a bijection. Since $M\approx\omega$, $Y$ is a countably infinite subset of $[X]^{<\omega}$, witnessing that $[X]^{<\omega}$ is D-infinite.

\item[(3)$\Rightarrow$(1)]
It follows directly from the definitions that if $X$ is \hd{}-infinite, then it is also \hd{2}-infinite.

\item[(4)$\Rightarrow$(1)]
It is readily apparent from the definitions that every \hd{}-infinite set must be \hb{}-infinite, and in particular also \hb{4}-infinite.

\item[(2)$\Rightarrow$(4)]
Suppose $X$ is \hb{4}-infinite, and consider the colouring $c:[X]^{<\omega}\longrightarrow 2$ given by $c(x)=\lfloor\log_2|x|\rfloor\mod 2$. By assumption, there exists an infinite set $Y\subseteq[X]^{<\omega}$ such that $\fs_\leq 4(Y)$ is monochromatic. The following lemma will allow us to finish the proof. To state the lemma, start by defining a function $F:\omega\times\omega\longrightarrow\omega$, by recursion on the second parameter, by $F(n,0)=4$ and $F(n,k+1)=2^n(R(F(n,k))-1)+2$, where $R(m)$ denotes the Ramsey number for obtaining a monochromatic complete $m$-graph from two colours\footnote{That is, $R(m)$ is the least integer $M$ such that for every colouring of $[M]^2$ with two colours, there exists a subset $Z\subseteq M$ with $|Z|=m$ and such that all elements of $[Z]^2$ have the same colour.}. Note that, for every $n$ and $k$, the number $F(n,k)\geq 4$.

\begin{lemma}\label{theonethatusesramsey}
Suppose that $Y\subseteq[X]^{<\omega}$ is such that $\fs_{\leq 4}(Y)$ is monochromatic for the colouring $c$ (as defined above). Then, for every $n\in\omega$, there are less than $F(n,n)$ many $y \in Y$ such that $|y|=n$.
\end{lemma}

To see how to finish our proof from this lemma, note that, if $Y\subseteq[X]^{<\omega}$ is monochromatic for $c$, then the lemma implies that, for every $n<\omega$, the set $y_n=\bigcup\{y\in Y\big||y|=n\}$ is finite (as it is a finite union of finite sets). Hence, the sequence $\langle y_n\big|n<\omega\rangle$ is a sequence of elements of $[X]^{<\omega}$; notice that this sequence has elements of arbitrarily large cardinality (since for infinitely many $n$ there is at least a $y$ with $|y|=n$, and hence $|y_n|\geq n$) and therefore one can extract an injective subsequence from it. This shows that $[X]^{<\omega}$ is D-infinite, and we are done.

\begin{proof}[Proof of Lemma~\ref{theonethatusesramsey}]
Suppose that $Y\subseteq[X]^{<\omega}$ is a set such that $\fs_{\leq 4}(Y)$ is monochromatic for $c$, and assume without loss of generality that every element of $Y$ has cardinality $n$. Working towards a contradiction, we further assume that $|Y|\geq F(n,n)$ and we will proceed to prove, by induction on $k\leq n$, that one can find at least $F(n,n-k)$ many distinct elements of $Y$ such that any two of them have an intersection of cardinality greater than $k$. Once we manage to prove that, the contradiction will be apparent, since in particular we will have $F(n,n-n)=4>2$ distinct elements of $Y$ whose intersection has cardinality greater than $n$, which is impossible since these elements were assumed to have themselves cardinality $n$.

The case $k=0$ is easy, as it amounts to showing that no two elements of $Y$ can be disjoint. This follows directly from Lemma~\ref{colourlog2}.

Now suppose that the result holds for $k$. Assume that $Y$ has been thinned out so that it only contains the $F(n,n-k)$ many elements, guaranteed to exist by induction hypothesis, such that any two of them have an intersection of cardinality at least $k$. We pick an arbitrary $y_0\in Y$ and we notice that, for each of the remaining $F(n,n-k)-1=2^n(R(F(n,n-(k+1)))-1)+1$ many $y\in Y$, the intersection $y\cap y_0\subseteq y_0$ can be one of $2^n$ many possibilities; hence by the pigeonhole principle there is a fixed $r$ and a set $Y'\subseteq Y\setminus\{y_0\}$ such that $|Y'|=R(F(n,n-(k+1)))$ and $(\forall y\in Y')(y\cap y_0=r)$, with $|r|\geq k$ by induction hypothesis. Now define a colouring on pairs of elements of $Y'$ by $d(y,z)=1$ if and only if $(y\setminus r)\cap(z\setminus r)\neq\varnothing$. By Ramsey's theorem, we can find a further $Y''\subseteq Y'$ such that $|Y''|=F(n,n-(k+1))$ and such that all pairs of elements from $Y''$ receive the same colour from $d$. We claim that this colour has to be $1$; to see this, assume by contradiction that it is not. Since $F(n,n-(k+1))\geq 4$, one can pick distinct $y_1,y_2,y_3,y_4\in Y''$ and our assumption about the colour means that the sets $y_i\setminus y_0$, $1\leq i\leq 4$, are pairwise disjoint; now, since $y_i\cap y_0=r$ for $1\leq i\leq 4$, the conclusion is that for any $1\leq i<j\leq 4$, we have $y_i\cap y_j=r$. This means that $y_1\bigtriangleup y_2$ and $y_3\bigtriangleup y_4$ are disjoint, both belong to $\fs_{\leq 4}(Y)$, and they both have cardinality $2n-|r|$. By Lemma~\ref{colourlog2}, $(y_1\bigtriangleup y_2)\bigtriangleup(y_3\bigtriangleup y_4)$ must have a different colour than $y_1\bigtriangleup y_2$. But this is a contradiction, since $y_1\bigtriangleup y_2\bigtriangleup y_3\bigtriangleup y_4\in\fs_{\leq 4}(Y)$. This contradiction shows that, in fact, the colour that all pairs from $Y''$ share must be $1$. This means that any two elements $y,w\in Y''$ must have an intersection that properly contains $r$, and thus this intersection has cardinality strictly greater than $|r|\geq k$, and we are done proving the inductive step.

The end of the above inductive proof finishes the proof of the Lemma.\qedhere$_{\text{Lemma~\ref{theonethatusesramsey}}}$
\end{proof}
\end{description}
\qedhere$_{\text{Theorem~\ref{equivalencefour}}}$
\end{proof}

\begin{corollary}
In $\zf$, the notions of \hd{}-finite and \hb{}-finite are equivalent, and either of them is equivalent to \hd{n}-finite for any $n\in\mathbb N\setminus\{1\}$ and also to \hb{n} finite for any $n\in\mathbb\setminus\{1,2,3\}$.
\end{corollary}

In other words, for any set $X$, in $\zf$, the ``pairwise disjoint'' version of Hindman's theorem (even for only two summands) at $X$ is equivalent to the ``Boolean'' version of the theorem (even for only four summands) at $X$. Note that this provides us with a way of answering Leader's question in the affirmative: at least if one phrases Hindman's theorem in the pairwise disjoint version, for infinite sets in $\zf$ it suffices to know that the theorem is satisfied for up to two summands to conclude that the full theorem is satisfied. On the other hand, we will see in the next section that \hb{2}-finiteness is not equivalent to \hb{}-finiteness; this provides us with a different way of answering Leader's question, this time in the negative: if one considers Hindman's theorem for Boolean groups in $\zf$, it is possible for Hindman's theorem to hold for up to two summands without the full version of the theorem being satisfied.

Since our jungle of different flavours of Hindman finiteness has collapsed to at most three non-equivalent notions, we introduce the following simplified notations.

\begin{definition}
Let $X$ be a set,
\begin{enumerate}
\item we will say that $X$ is \hi{2}-finite if it is what used to be called \hb{2}-finite,
\item we will say that $X$ is \hi{3}-finite if it is what used to be called \hb{3}-finite, and
\item we will say that $X$ is H-finite if it is what used to be called \hb{}-finite (equivalently, what used to be called \hb{n}-finite for $n\in\mathbb N\setminus\{1,2,3\}$; also equivalently, what used to be called \hd{n}-finite for $n\in\mathbb N\setminus\{1\}$, or, also equivalently, what used to be called \hd{}-finite).
\end{enumerate}
\end{definition}

We now proceed to obtain a couple of implications between some of these notions of Hindman finiteness and some of the other notions of finiteness that have been introduced before.

\begin{theorem}
Every C-finite set is H-finite.
\end{theorem}

\begin{proof}
Suppose that $X$ is an H-infinite set. By Theorem~\ref{equivalencefour}, this implies that $[X]^{<\omega}$ is D-infinite. In particular, $\wp(X)$ is D-infinite, but this is, by~\cite[Lemma 4.11]{herrlich-choice}, equivalent to the fact that $X$ is C-infinite, which finishes the proof.
\end{proof}

\begin{theorem}\label{h2finimpliesr2fin}
Every \hi{2}-finite set is \ra{2}-finite.
\end{theorem}

\begin{proof}
Suppose that $X$ is \ra{2}-infinite, and let $c:[X]^{<\omega}\longrightarrow 2$ be a colouring. Then there is an infinite $Z\subseteq X$ such that $[Z]^2$ is monochromatic for the colouring $c\upharpoonright[X]^2$. Fix $z\in Z$ and define $Y=\left\{\{y,z\}\big|y\in Z\setminus\{z\}\right\}$. Notice that every element of $Y$ is a doubleton whose elements belong to $Z$; notice also that, if $\{y,z\}\in Y$ and $\{y',z\}\in Y$ are two distinct elements, then $\{y,z\}\bigtriangleup\{y',z\}=\{y,y'\}$ is also a doubleton whose elements belong to $Z$. This shows that $\fs_{\leq 2}(Y)\subseteq[Z]^2$, and therefore $\fs_{\leq 2}(Y)$ is monochromatic for $c$; thus $X$ is \hi{2}-infinite.
\end{proof}

\begin{figure}
\centerline{\xymatrix{
 & & \text{Finite} \ar@{=>}[dl] \ar@{=>}[d] \ar@{=>}[ddrr] & & \\
 & \text{A-finite} \ar@{=>}[d] \ar@{=>}[ddl] & \text{H}_2\text{-finite} \ar@{=>}[d] \ar@{=>}[dr] & & \\
 & \text{C-finite} \ar@{=>}[d] \ar@{=>}[dr] & \text{H}_3\text{-finite} \ar@{=>}@[red][d] & \text{R}^2\text{-finite} \ar@{=>}[dddl] & \text{R}^n\text{-finite}\ \ (n\geq 3) \ar@{=>}[dddll] \\
\text{B-finite} \ar@{=>}[ddrr] & \text{E-finite} \ar@{=>}[ddr] & \text{H-finite} \ar@{=>}[dd] & & \\
 & & & & \\
 & & \text{D-finite.} & & \\ 
}}
\caption{Implications provable in $\zf$ between the finiteness classes considered in this paper.}
\label{fig1}
\end{figure}
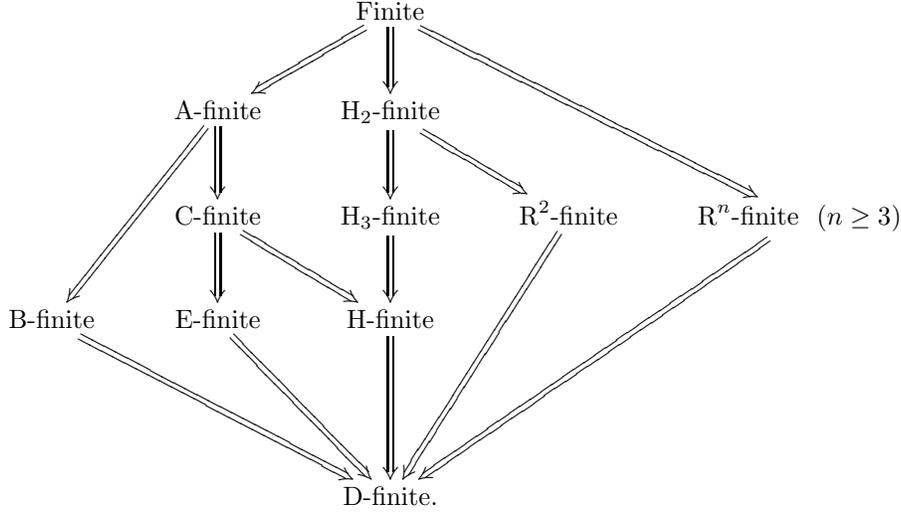

Figure~\ref{fig1} summarizes the $\zf$-provable implications between the different finiteness classes that have been considered in this paper; in the next section we will prove that, for the most part, no other arrows can be added to this diagram.

\section{Independence Results}\label{Sect:Independence}

We claim that the diagram in Figure~\ref{fig1} contains essentially all possible implications that are provable in $\zf$. To show that there are no others, we need to exhibit independence proofs by obtaining models with counterexamples to the remaining implications.

The method for independence proofs used in this paper is that of Fraenkel--Mostowski permutation models. In this method, one works in a set theory that allows atoms, elements of the set-theoretic universe that are not sets themselves. More formally, the theory $\zfa$ is a theory in a first-order language with non-logical symbols for equality and membership, as well as two special constant symbols $\varnothing$ and $A$; the first such symbol's intended meaning is for the empty set, whereas the second is supposed to stand for the set of all atoms. The axioms for $\zfa$ are just the usual axioms of $\zf$, appropriately modified to allow for the existence of atoms (for example, extensionality is no longer stated for any two objects, but only for any two objects that are not elements of $A$), plus an empty-set axiom stating $\neg(\exists x)(x\in\varnothing)$ and an atoms axiom stating that $(\forall z)(z\in A\iff(z\neq\varnothing\wedge\neg(\exists x)(x\in z)))$, that is, the set of all atoms contains precisely those objects that are not the empty set yet they have no elements. This theory (even with the extra assumption that $A$ is infinite) is known to be consistent relative to $\zf$ (see~\cite[p. 51, Problem 1]{jech-choice}).

When working in such a theory, every permutation of the atoms, $\pi\in S_A$, induces an automorphism of the set-theoretic universe defined recursively by $\pi(x)=\{\pi(y)\big|y\in x\}$. Given a subgroup $G$ of the full permutation group $S_A$ of the set of atoms, and given an arbitrary set $x$, a \textbf{support} for $x$ (relative to $G$, which will most of the time be implicit) is a set $E\subseteq A$ such that $(\forall\pi\in G)((\forall a\in E)(\pi(a)=a)\Rightarrow\pi(x)=x)$. Note that supports for a set $x$ are not unique, since if $E$ is a support for $x$ then so is every $F\supseteq E$. Note also that if $x$ is a {\it pure set} ---that is, if the transitive closure of $x$ does not intersect $A$--- then $\pi(x)=x$ for all $\pi\in S_A$; in particular pure sets always admit $\varnothing$ as a support.

All of the models considered in this paper are constructed as follows: one starts by assuming $\zfa+\ac$, in a universe where the set of atoms $A$ is countable, and one fixes a subgroup $G\leq S_A$. Then one calls a set {\it symmetric} if it has a support that is finite (relative to the subgroup $G$), and one collects the class $\hs$ of all hereditarily symmetric sets. Then~\cite[Theorem 4.1]{jech-choice}, the class $\hs$ will be a transitive model of $\zfa$ (containing $A$ and all pure sets) that, in general, will not satisfy $\ac$. For a more general treatment of Fraenkel--Mostowski permutation models, we direct the reader to~\cite[Chapter 4]{jech-choice}.

Given a set $X$, recursively define $\wp^\alpha(X)$ by $\wp^0(X)=X$, $\wp^{\xi+1}(X)=\wp(\wp^\xi(X))$ and $\wp^\xi(X)=\bigcup_{\eta<\xi}\wp^\eta(X)$ for $\xi=\bigcup\xi$. Note that, for each of the finiteness classes $\mathscr F$ considered in this paper, membership of a set $X$ to said class is a statement that can be decided by $\wp^\alpha(X)$, in the sense that $(X\in\mathscr F)\iff(X\in\mathscr F)^{\wp^\alpha(X)}$, for a finite---and in fact, a very small---ordinal $\alpha$. For example, $X$ is H-finite if and only if for every colouring $c:[X]^{<\omega}\longrightarrow 2$, there is an infinite $Y\subseteq[A]^{<\omega}$ such that $\fs(Y)$ satisfies an absolute statement that has $c$ as a parameter. Such a colouring $c$, if we define functions as usual using Kuratowski's ordered pairs, would be an element of $\wp^5(X)$; the relevant $X$ is simply an element of $\wp^2(X)$, and $\fs(Y)\in\wp^2(X)$ as well. Furthermore, the fact that $Y$ is infinite can be phrased using Tarski's definition of finiteness (every nonempty subset of $\wp(Y)$ has a $\subseteq$-minimal element), and so $Y$ is infinite if and only if $(Y\text{ is infinite})^{\wp^3(Y)}$. The conclusion of this is that a set $X$ is H-finite if and only if $(X\text{ is H-finite})^{\wp^5(X)}$. A similar analysis can be done with all the other finiteness classes considered in this paper.

Given all of the above, the template for our independence proofs will be as follows: in order to prove that X-finite does not imply Y-finite, we will describe a suitable $G\leq S_A$ in such a way that, in the class $\hs$, there is a set $Z$ which is X-finite but Y-infinite. In all of the proofs within this paper, the set $Z$ will be either the set of all atoms, or some carefully defined set of sets of atoms. Therefore the existence of an X-finite, Y-infinite set will be a statement satisfied not only by the Fraenkel--Mostowski model, but in fact by the structure $\wp^{11}(A)$, where $A$ is the set of atoms (the number 11 is overkill, to be safe; the reader that wishes to do so may more accurately bound it). The Jech--Sochor transfer theorem~\cite[Theorem 6.1]{jech-choice} (see also~\cite{jech-sochor}) states that, for every ordinal number $\alpha$ and every subgroup $G\leq S_A$, there exists a transitive model $\mathbf{W}$ of $\zf$ (concretely, $\mathbf{W}$ is a {\it symmetric extension} of the universe, i.e. a subclass of an extension of the universe by forcing) and an element $B\in\mathbf{W}$ such that the structure $(\wp^\alpha(A))^{\hs}$ is $\in$-isomorphic to the structure $(\wp^\alpha(B))^{\mathbf{W}}$, and therefore the existence of an X-finite, Y-infinite set within $\wp^{11}(A)$ in a Fraenkel--Mostowski model implies the existence of such a set in a model of $\zf$. Therefore, once we have obtained a set with the required properties in some Fraenkel--Mostowski permutation model, we will automatically know that our task is complete (by implicitly invoking the Jech--Sochor theorem).

Having outlined our strategy for independence proofs, we now proceed to describe the relevant permutation models of $\zfa$, as well as the relevant sets within them.

\subsection{The first Fraenkel model}

The first Fraenkel model is the model of $\zfa$ that arises from considering the class $\hs$ with respect to the full permutation group $G_1=S_A$ of the set of all atoms. It is well-known, and also easy to prove~\cite[p. 52, Problem 7]{jech-choice}, that in the first Fraenkel model the set $A$ of atoms is amorphous, that is, both infinite and A-finite (in fact, $A$ is strongly amorphous). It turns out that this set is also \r-infinite, for all $n\geq 2$, as we argue below.

\begin{proposition}\label{firstfrankelrinfinite}
For each $n\in\mathbb N\setminus{2}$, the set $A$ of atoms in the first Fraenkel model is \r-infinite.
\end{proposition}

\begin{proof}
Suppose that $c:[A]^n\longrightarrow 2$ is a colouring in $\hs$. Then there is a finite $F\subseteq A$ which is a support for $c$. We claim that the (infinite) set $A\setminus F$, which clearly belongs to $\hs$ (as it has $F$ as a support), is monochromatic for $c$. To see this, let $x,y\in[A\setminus F]^n$. It is easy to get a permutation $\pi\in S_A$ which fixes each element of $F$ and such that $\pi[x]=y$. Since $\pi$ fixes $F$, which in turn is a support for $c$, we must have that $\pi(c)=c$. In particular, if $c(x)=i$ then we have that
\begin{equation*}
(y,i)=(\pi(x),\pi(i))=\pi((x,i))\in\pi(c)=c,
\end{equation*}
which shows that $c(y)=i$ as well. Since $y\in[A\setminus F]^n$ was arbitrary, we conclude that the set $[A\setminus F]^n$ is monochromatic in colour $i$, and we are done.
\end{proof}

As a corollary of Proposition~\ref{firstfrankelrinfinite}, we can also conclude that the set of atoms in the first Fraenkel model is \hi{2}-infinite, by Theorem~\ref{h2finimpliesr2fin}. On the other hand, we proceed to argue below that this set is \hi{3}-finite.

\begin{proposition}\label{firstfrankelh3finite}
The set $A$ of atoms in the first Fraenkel model is \hi{3}-finite.
\end{proposition}

\begin{proof}
Within the class $\hs$, define the colouring $c:[A]^{<\omega}\longrightarrow 2$ given by
\begin{equation*}
c(x)=\begin{cases}
0;\text{ if }|x|\equiv0\mod 4\text{ or }|x|\equiv 1\mod 4, \\
1;\text{ otherwise.}
\end{cases}
\end{equation*}
Note that every permutation $\pi\in S_A$ preserves the cardinality of every finite set of atoms, and therefore the colouring $c$ is indeed hereditarily symmetric, as it has $\varnothing$ as a support. We proceed to prove that no (hereditarily symmetric) infinite set $Y\subseteq [A]^{<\omega}$ can possibly be such that $\fs_{\leq 3}(Y)$ is monochromatic for $c$. To see this, let $Y\subseteq[A]^{<\omega}$ be a hereditarily symmetric infinite set, and let $F\subseteq A$ be a finite set that is a support for $Y$. Since $F$ is finite and $Y$ is infinite, we can find a $y\in Y$ such that $y\not\subseteq F$. Fix an atom $a\in y\setminus F$ and fix two distinct atoms $b,c\in A\setminus(y\cup F)$. Let $\pi\in S_A$ be the transposition swapping $a$ and $b$, and let $\sigma\in S_A$ be the transposition swapping $a$ and $c$. Since both $\pi$ and $\sigma$ pointwise fix all elements of $F$, we have that $\pi(Y)=Y$ and $\sigma(Y)=Y$; therefore
\begin{equation*}
z:=(y\setminus\{a\})\cup\{b\}=\pi(y)\in\pi(Y)=Y,
\end{equation*}
and similarly
\begin{equation*}
w:=(y\setminus\{a\})\cup\{c\}=\sigma(y)\in\sigma(Y)=Y.
\end{equation*}
Thus $y\bigtriangleup z\bigtriangleup w\in\fs_{\leq 3}(Y)$. Note, however, that 
\begin{equation*}
|y\bigtriangleup z\bigtriangleup w|=|y\bigtriangleup((y\cup\{b\})\setminus\{a\})\bigtriangleup((y\cup\{c\})\setminus\{a\})|=|y\cup\{b,c\}|=|y|+2,
\end{equation*}
which implies that $c(y)\neq c(y\bigtriangleup z\bigtriangleup w)$ by definition of $c$, and hence $\fs_{\leq 3}(Y)$ is not monochromatic. This shows that $A$ is \hi{3}-finite, finishing the proof.
\end{proof}

We summarize below the conclusions that we can reach, based on the previous two propositions, about implications between our finiteness classes.

\begin{corollary}\label{conclusionsfirstfrankel}
Modulo the theory $\zf$, we have that, in general,
\begin{enumerate}
\item A-finite does not imply \hi{2}-finite nor \ra{n}-finite ($n\geq 2$);
\item therefore, neither of B-finite, C-finite or E-finite imply \hi{2}-finite nor \ra{n}-finite ($n\geq 2$) either;
\item \hi{3}-finite does not imply \hi{2}-finite;
\item for no $n\in\mathbb N\setminus\{1\}$ does \hi{3}-finite imply \r-finite;
\item consequently, it is also the case that neither H-finite nor D-finite imply \r-finite whenever $n\geq 2$.
\end{enumerate}
\end{corollary}

\subsection{Variations of the second Fraenkel model}

The second Fraenkel model is constructed as follows: one starts by partitioning the set of atoms $A$ into a countable union of two-element sets, that is, we write $A=\bigcup_{n<\omega}P_n$, where the $P_n$ are disjoint and $(\forall n<\omega)(|P_n|=2)$. We let 
\begin{equation*}
G_2=\{\pi\in S_A\big|(\forall n<\omega)(\pi[P_n]=P_n)\},
\end{equation*}
in other words, $G_2$ is the subgroup of all permutations of the set of atoms that setwise fix each $P_n$, although they might or might not flip the elements of $P_n$ for each individual $n$. The second Fraenkel model is the class $\hs$ of all hereditarily symmetric sets with respect to the group $G_2$. It is well-known, as well as easy to see~\cite[Section 4.4]{jech-choice}, that the set of atoms $A$ in the second Fraenkel model is a Russell set, and therefore \r-finite for every $n\geq 2$ by Proposition~\ref{russellsetrfinite}. Russell sets are easily seen to be B-finite, and they are also E-finite~\cite[Theorem 2.2]{herrlich-howard-tachtsis}; therefore the set $A$ belongs to both of these finiteness classes. Below we explore the H-finiteness of this set.

\begin{proposition}
The set of atoms in the second Fraenkel model is H-infinite.
\end{proposition}

\begin{proof}
Note that the set $\{P_n\big|n<\omega\}$ is hereditarily symmetric with respect to $G_2$, and so is the mapping $n\longmapsto P_n$. Thus the set $\{P_n\big|n<\omega\}\subseteq [A]^{<\omega}$ is countable from the perspective of the permutation model, hence witnessing that $[A]^{<\omega}$ is D-infinite. Therefore, by Theorem~\ref{equivalencefour}, $A$ must be H-infinite.
\end{proof}

Thus, in the second Fraenkel model, the set of atoms is at the same time H-infinite (and consequently C-infinite, A-infinite, \hi{3}-infinite and \hi{2}-infinite) and \r-finite for all $n\geq 2$, E-finite, and B-finite. We summarize this conclusion below.

\begin{corollary}\label{cor-second-frankel}
Modulo the theory $\zf$, we have that, in general,
\begin{enumerate}
\item for no $n\in\mathbb N\setminus\{1\}$ does \r-finite imply H-finite;
\item therefore for no $n\in\mathbb N\setminus\{1\}$ does \r-finite imply finite;
\item D-finite does not imply H-finite;
\item E-finite does not imply H-finite;
\item B-finite does not imply H-finite;
\item consequently, it is also the case that neither of D-finite, E-finite, B-finite or \r-finite for some $n\geq 2$ imply \hi{3}-finite nor \hi{2}-finite.
\end{enumerate}
\end{corollary}

In particular, we have now established that, for each $n\geq 2$, the notions of \r-finite and H-finite are independent (i.e. neither implies the other), as are the notions of \r-finite and \hi{3}-finite. Furthermore, the implication from \hi{2}-finite to \ra{2}-finite is not reversible. Also, for each $n\geq 2$, the notion of \r-finite is independent from each of the notions of A-finite and C-finite.

So far, none of the sets under consideration has been E-infinite, hence we now proceed to construct one such set.

\begin{definition}
Working in the second Fraenkel model, we will use the letter $\mathcal B$ to denote the set $\bigcup_{n<\omega}\mathcal B_n$, where
\begin{equation*}
\mathcal B_n=\left\{B\subseteq A\big|(\forall i\leq n)(|B\cap P_n|=1)\wedge(\forall i>n)(B\cap P_n=\varnothing)\right\}.
\end{equation*}
In other words, $\mathcal B_n$ consists of all selectors of the finite family of pairs $\{P_i\big|i\leq n\}$; consequently $\mathcal B$ consists of all finite nonempty selectors of some initial segment of the indexed family $\{P_n\big|n<\omega\}$.
\end{definition}

Notice that each individual element $B\in\mathcal B_n$ has $\bigcup_{i\leq n}P_n$ as a support, and therefore belongs to the second Fraenkel model. Trivially, each $\mathcal B_n$, as well as $\mathcal B$ admit $\varnothing$ as a support, and thus these sets belong to the second Fraenkel model as well.

\begin{proposition}
In the second Fraenkel model, the set $\mathcal B$ is E-infinite and \r-finite for all $n\in\mathbb N\setminus\{1\}$.
\end{proposition}

\begin{proof}
To see that $\mathcal B$ is E-infinite, we take the proper subset $\mathcal C=\{B\in\mathcal B\big||B|\geq 2\}=\bigcup_{n\geq 1}\mathcal B_n$ and observe that, for each $n<\omega$, the function $f_n:\mathcal B_{n+1}\longrightarrow\mathcal B_n$ given by $f_n(B)=B\setminus P_{n+1}$ is a surjection. Hence if we let $f=\bigcup_{n<\omega}f_n$, we will have that $f:\mathcal C\longrightarrow\mathcal B$ is a surjection witnessing that $\mathcal B$ is E-infinite.

Now, to see that $\mathcal B$ is \r-finite, we will show that it admits a partition into finite cells such that no infinite subfamily of the partition carries a choice function. The partition in question is given by $\{\mathcal B_n\big|n<\omega\}$. If $X\subseteq\omega$ was an infinite set and $g:X\longrightarrow\mathcal B$ was a choice function within the model, with some finite support $F\subseteq A$, then letting $n\in\mathbb N$ be big enough that $P_n\cap F=\varnothing$ we would have that the permutation $\pi$ transposing the pair $P_n$ would need to setwise fix $g$, while at the same time altering the value of $g(n)$, which is a contradiction. Hence, the partition is as claimed, and so $\mathcal B$ is \r-finite for all $n\geq 2$ by Proposition~\ref{russellsetrfinite}
\end{proof}

\begin{corollary}
Modulo the theory $\zf$, in general \r-finite (for any given $n\in\mathbb N\setminus\{1\}$) does not imply E-finite. Hence, for each $n\geq 2$, the notions of \r-finite and E-finite are independent.
\end{corollary}

Producing sets that are B-infinite is a subtle endeavour that we leave for the next subsection; a similar comment applies to producing sets that are infinite but \hi{2}-finite. So far we have established that the notions of \r-finite are independent from A-, C-, E-, H- and \hi{3}-finite; one could say that our picture of \r-finiteness is as complete as it could be at this moment. Unfortunately, the set $\mathcal B$ considered above is H-infinite (as witnessed by the countable sequence $\langle\mathcal B_n\big|n<\omega\rangle$ of finite subsets of $\mathcal B$), so in order to complete our picture of H-finiteness we need to consider other sets in other models. We will explore the following model, which is a variation on the idea of the second Fraenkel model.

\begin{definition}
We define the \textbf{$\omega$-Fraenkel model} by means of the following construction. Start by partitioning the set of atoms $A$ as $\bigcup_{m<\omega}A_m$, where each $A_m$ is an infinite set. Then consider the group of permutations $G_3\leq S_A$ given by
\begin{equation*}
G_3=\{\pi\in S_A\big|(\forall m<\omega)(\pi[A_m]=A_m)\},
\end{equation*}
and our model will be the class $\hs$ of sets that are hereditarily symmetric with respect to $G_3$.
\end{definition}

Thus, the $\omega$-Fraenkel model is built by following the same idea that is used in the second Fraenkel model, but with infinite sets instead of pairs.

\begin{proposition}\label{omegafrankel}
In the $\omega$-Fraenkel model, the set $A$ of atoms is C-infinite, \hi{3}-finite, and \r-infinite (for each $n\geq 2$).
\end{proposition}

\begin{proof}
It is easy to see that the sequence mapping $m$ to $A_m$ is hereditarily symmetric (admitting $\varnothing$ as a support), which witnesses that $A$ is C-infinite. Notice that the permutations of the group $G_3$ are free to permute each of the elements within a fixed $A_m$ arbitrarily, and hence in the $\omega$-Fraenkel model each $A_m$ behaves much like the set of atoms in the first Fraenkel model. Thus, if $c:[A]^n\longrightarrow 2$ is a colouring within the model, say with support $F$, then picking an $m$ large enough that $A_m\cap F=\varnothing$ it is easy to see that $A_m$ is monochromatic for $c$, since given any two $n$-tuples $x,y\subseteq A_m$ we can always find a permutation $\pi$ such that $\pi[x]=y$ and fixing everything else; this permutation will fix $F$ and hence $(y,c(x))=(\pi(x),\pi(c(x)))=\pi(x,c(x))\in\pi(c)=c$, showing that $c(y)=c(x)$ and therefore that $A$ is \r-infinite. Similarly, if we define $c:[A]^{<\omega}\longrightarrow 2$ by $c(x)=1$ iff either $|x|\equiv 0\mod 4$ or $|x|\equiv 1\mod 4$, and $Y\subseteq A$ is an infinite set in our model, say with support $F$, find a $y\in Y$ with $y\setminus F\neq\varnothing$, let $a\in y\setminus F$ and let $m$ be such that $a\in A_m$. Picking $b,c\in A_m\setminus(F\cup y)$ and letting $\pi,\sigma$ be the permutations that transpose $a$ and $b,c$ respectively, we observe that $\pi,\sigma\in G_3$ and thus $z=\pi(y)=(y\setminus\{a\})\cup\{b\}$ and $w=\sigma(y)=(y\setminus\{a\})\cup\{c\}$ both belong to $Y$, which implies that $y\bigtriangleup z\bigtriangleup w\in\fs_{\leq 3}(Y)$ while at the same time $|y\bigtriangleup z\bigtriangleup w|=|y\cup\{b,c\}|=|y|+2$, which means that $c(y)\neq c(y\bigtriangleup z\bigtriangleup w)$ and thus $\fs_{\leq 3}(Y)$ is not monochromatic for $c$, finishing the proof that $A$ is \hi{3}-finite.
\end{proof}

\begin{corollary}
In general, modulo the theory $\zf$, \hi{3}-finite (and consequently also H-finite) does not imply C-finite nor A-finite. In particular, the implication C-finite$\Rightarrow$H-finite (and consequently also the implication A-finite$\Rightarrow$H-finite) is not reversible.
\end{corollary}

We now use an idea already used in the second Fraenkel model to further investigate the relation between H-finite and E-finite.

\begin{definition}
Working in the $\omega$-Fraenkel model, we define the set $\mathscr B=\bigcup_{n<\omega}\mathscr B_n$, where
\begin{equation*}
\mathscr B_n=\{B\subseteq A\big|(\forall i\leq n)(|B\cap A_n|=1)\wedge(\forall i>n)(B\cap A_n=\varnothing)\}.
\end{equation*}
\end{definition}

\begin{proposition}
In the $\omega$-Fraenkel model, the set $\mathscr B$ is E-infinite and H-finite.
\end{proposition}

\begin{proof}
The proof that $\mathscr B$ is E-infinite is very much like the proof that $\mathcal B$ is E-infinite in the second Fraenkel model. Namely, consider the set $\mathscr C=\bigcup_{n\geq 1}\mathscr B_n=\{B\in\mathscr B\big||B|\geq 2\}$ and let $f:\mathscr C\longrightarrow\mathscr B$ be the mapping sending each $B\in\mathscr B_n$ to $B\setminus A_n$; this maps the proper subset $\mathscr C$ of $\mathscr B$ onto $\mathscr B$ and hence witnesses that $\mathscr B$ is E-infinite.

Now, to see that $\mathscr B$ is H-finite, assume the opposite. This means that there is a hereditarily symmetric countable injective sequence $\langle F_n\big|n<\omega\rangle$ of finite subsets of $\mathscr B$; if the support of that sequence is $F$, then we have that for every $\pi\in G_3$ fixing $F$ pointwise, it must be the case that $\{(n,F_n)\big|n<\omega\}=\pi\left(\{(n,F_n)\big|n<\omega\}\right)=\{\pi(n,F_n)\big|n<\omega\}=\{(\pi(n),\pi(F_n))\big|n<\omega\}=\{(n,\pi(F_n))\big|n<\omega\}$. This shows that, if $\pi\in G_3$ fixes $F$ pointwise, then it must fix also each $F_n$ individually. However, since the sequence is injective and $F$ is finite, there must be an $n<\omega$ such that $\left(\bigcup F_n\right)\not\subseteq F$; now if we let $a\in\left(\bigcup F_n\right)\setminus F$, $m<\omega$ such that $a\in A_m$, and $b\in A_m\setminus\left(\left(\bigcup F_n\right)\cup F\right)$, then the transposition $\pi\in G_3$ exchanging $a$ and $b$ will be a permutation fixing $F$ pointwise but moving $F_n$. This contradiction shows that $\mathscr B$ must be H-finite, and we are done.
\end{proof}

\begin{corollary}
Modulo the theory $\zf$, in general H-finite does not imply E-finite. Consequently, the notions of H-finite and E-finite are independent.
\end{corollary}

We now seem to know all that there is to know regarding implications, or lack thereof, between H-finite and the other notions of finiteness under consideration (except for B-finite). We do not seem to have such a complete picture regarding the notion of \hi{3}-finite. In fact, we were not able to determine whether H-finite is equivalent to \hi{3}-finite. The following construction shows that this question might indeed be very hard. Start with a set of atoms indexed by $\omega\times\omega$, $A=\{a_{i,j}\big|i,j<\omega\}$, and consider the permutation group
\begin{equation*}
G_4=\left\{\pi\in S_A\big|(\exists\sigma,\rho\in S_\omega)(\forall i,j<\omega)\left(\pi(a_{i,j})=a_{\sigma(i),\rho(j)}\right)\right\}.
\end{equation*}

\begin{proposition}
Consider the set $A$ of atoms in the class $\hs$ of hereditarily symmetric sets with respect to the group $G_4$. This set satisfies that:
\begin{enumerate}
\item it is \hi{3}-finite (and consequently also H-finite),
\item for every colouring $c:[A]^{<\omega}\longrightarrow 2$ mapping any two sets of the same cardinality to the same colour, there exists an infinite $Y\subseteq[A]^{<\omega}$ such that $\fs_{\leq 3}(Y)$ is $c$-monochromatic.
\end{enumerate}
\end{proposition}

\begin{proof}\hfill
\begin{enumerate}
\item Within our model, we let $d:[A]^{<\omega}\longrightarrow\omega$ be given by
\begin{equation*}
d(x)=\left|\{i<\omega\big|(\exists j<\omega)(a_{i,j}\in x\}\right|+\left|\{j<\omega\big|(\exists i<\omega)(a_{i,j}\in x\}\right|,
\end{equation*}
(so that $d(x)$ measures how many different ``rows'' and ``columns'' are intersected by $x$), and now we define $c:[A]^{<\omega}\longrightarrow 2$ by
\begin{equation*}
c(x)=\begin{cases}
0;\text{ if }d(x)\equiv 0\mod 4\text{ or }d(x)\equiv 1\mod 4, \\
1;\text{ otherwise.}
\end{cases}
\end{equation*}
Notice that, if $\pi\in G_4$ is arbitrary, then for every $x\in A^{<\omega}$ we have that $d(x)=d(\pi(x))$. This means that $d$ is (hereditarily) symmetric (supported by the empty set) and hence so is $c$. Now we let $Y$ be an arbitrary infinite (symmetric) subset of $[A]^{<\omega}$, and let $F\subseteq A$ be a finite set supporting $Y$. Let $N$ be large enough that, if $a_{i,j}\in F$, then $i,j<N$. As $Y$ is infinite, we can pick an $x\in Y$ such that there exists an $a_{m,n}\in x$ with $N<m$ or $N<n$. Suppose that $N<m$ (the case $N<n$ is treated analogously), and let $\sigma,\sigma'\in S_\omega$ be two different transpositions exchanging $m$ with some numbers $\sigma(m),\sigma'(m)$ such that, if $a_{m',n'}\in x$, then $m'<\sigma(m)<\sigma'(m)$. Now we let $\pi,\pi'\in S_A$ be the permutations given by $\pi(a_{i,j})=a_{\sigma(i),j}$ and $\pi'(a_{i,j})=a_{\sigma'(i),j}$. Since $\pi$ and $\pi'$ pointwise fix $F$, we must have that $y:=\pi(x)$ and $z:=\pi'(x)$ belong to $Y$. Therefore 
\begin{equation*}
x\cup\{a_{\sigma(m),j}\big|a_{m,j}\in x\}\cup\{a_{\sigma'(m),j}\big|a_{m,j}\in x\}=x\bigtriangleup y\bigtriangleup z\in\fs_{\leq 3}(Y).
\end{equation*}
Notice that $x\bigtriangleup y\bigtriangleup z$ intersects exactly two rows more than $x$ does, and the same number of columns. In other words, $d(x\bigtriangleup y\bigtriangleup z)=d(x)+2$, and consequently $c(x)\neq c(x\bigtriangleup y\bigtriangleup z)$. Hence $\fs_{\leq 3}(Y)$ cannot be monochromatic for $c$, and thus $A$ is not \hi{3}-finite.

\item Let $c:[A]^{<\omega}\longrightarrow 2$ be such that, whenever $x,y\in [A]^{<\omega}$ have the same cardinality, we have $c(x)=c(y)$. Consider the colouring $d:\omega\longrightarrow 2$ defined by $d(n)=c(x)$, where $x\in[A]^n$ is arbitrary. We apply Schur's theorem (i.e. the smallest finitary version of Hindman's theorem) to the colouring $d\upharpoonright\{2n\big|n<\omega\}$ to obtain a monochromatic set of even numbers of the form $\{m,m',m+m'\}$, say on colour $i$. We let $k=m'/2$ and $n=m-k$, so that the monochromatic set above can be rewritten as $\{n+k,2k,n+3k\}$. Now consider the following subset of $[A]^{<\omega}$:
\begin{equation*}
Y=\left\{\{a_{1,1},\ldots,a_{1,n}\}\cup\{a_{i,n+1},\ldots,a_{i,n+k}\}\big|1<i<\omega\right\},
\end{equation*}
which admits $\{a_{1,1},\ldots,a_{1,n},a_{1,n+1},\ldots,a_{1,n+k}\}$ as a support and is therefore symmetric. Notice that, whenever $x,y,z\in Y$ are distinct, we have that $|x|=n+k$, $|x\bigtriangleup y|=2k$ and $|x\bigtriangleup y\bigtriangleup z|=n+3k$; this implies that $c(x)=d(n+k)=d(2k)=c(x\bigtriangleup y)=d(n+3k)=c(x\bigtriangleup y\bigtriangleup z)=i$. In other words, $\fs_{\leq 3}(Y)$ is monochromatic for $c$ in colour $i$.
\end{enumerate}
\end{proof}

Thus, although the set $A$ of atoms in this model is not \hi{3}-infinite, it is ``almost'' \hi{3}-infinite in the sense that it produces infinite monochromatic sets of the form $\fs_{\leq3}(Y)$ whenever the colouring in question is defined solely in terms of cardinality. This fact deserves highlighting, which we do below.

\begin{corollary}
It is consistent with $\zf$ that there exists an H-finite set satisfying that, for every colouring $c:[X]^{<\omega}\longrightarrow 2$ for which there exists $g:\omega\longrightarrow 2$ making the following diagram commutative

\centerline{\xymatrix{
[X]^{<\omega} \ar@{->}[r]^{|\cdot|} \ar@{->}[dr]_c & \omega \ar@{->}[d]^g \\
 & 2,
}}
one can find an infinite $Y\subseteq[X]^{<\omega}$ such that $\fs_{\leq 3}(Y)$ is monochromatic.
\end{corollary}

This fact essentially prevents us from proving that an arbitrary H-finite set $X$ needs to be \hi{3}-finite, since any such proof would presumably require us to define a ``bad'' colouring for $\fs_{\leq 3}$, but, in the absence of any assumptions about the structure on the arbitrary set $X$, it is hard to imagine how could one try to define this colouring, other than in terms of cardinality. Hence, we believe that the corollary above strongly suggests that one should try to find a model of $\zf$ with an H-finite but \hi{3}-infinite set, rather than trying to prove that H-finite implies \hi{3}-finite. The authors of this paper did not succeed in either of these two endeavours.

\subsection{Models with an ultrahomogeneous set of atoms}

Recall that a structure is said to be {\it ultrahomogeneous} if every isomorphism between two finite substructures can be extended to an automorphism of the whole structure. An idea that has proven to be very fruitful in the realm of permutation models is that of endowing the set of atoms with some sort of ultrahomogeneous structure and then taking the Fraenkel--Mostowski model with respect to the automorphism group of this structure.

A classical example of this approach is the model known as {\it Mostowski's linearly ordered model}. To construct this model, one first endows the set $A$ of atoms with a linear order $\leq$ in such a way that $(A,\leq)\cong(\mathbb Q,\leq_{\mathbb Q})$, where $\leq_{\mathbb Q}$ is the usual order on the rational numbers. Then we let $G_5$ be the automorphism group of $(A,\leq)$, and Mostowski's linearly ordered model is just the class $\hs$ of hereditarily symmetric sets with respect to the group $G_5$. This is a model that satisfies the linear ordering principle, stating that every set admits a linear order; we henceforth denote this principle by $\lo$. So this is a model in which every infinite set must be B-infinite, and so it is a model in which potentially we can find that some of our finiteness notions do not imply B-finiteness. A part of this can be seen in the following proposition.

\begin{proposition}\label{mostowskih3finite}
The set $A$ of atoms in Mostowski's linearly ordered model is \hi{3}-finite.
\end{proposition}

\begin{proof}
The proof is essentially the same as that of Proposition~\ref{firstfrankelh3finite} or~\ref{omegafrankel}: working in our model, we define the colouring $c:[A]^{<\omega}\longrightarrow 2$ by
\begin{equation*}
c(x)=\begin{cases}
0;\text{ if }|x|\equiv0\mod 4\text{ or }|x|\equiv 1\mod 4, \\
1;\text{ otherwise,}
\end{cases}
\end{equation*}
which can be done as $c$ has $\varnothing$ as a support. We now take an arbitrary infinite hereditarily symmetric set $Y\subseteq[A]^{<\omega}$ admitting a finite support $F\subseteq A$. Find a $y\in Y$ such that $y\not\subseteq F$, fix an atom $a\in y\setminus F$ and let $I$ be an interval in $(A,\leq)$ containing $a$ but not intersecting $F\cup y\setminus\{a\}$ (this can be done because the latter is a finite set). Now pick two distinct atoms $b,c\in I$, and note that, by ultrahomogeneity of the structure $(A,\leq)$, one can find automorphisms $\pi,\sigma$ of $(A,\leq)$ fixing $F\cup y\setminus\{a\}$ and such that $\pi(a)=b$ and $\sigma(a)=c$. Our assumption about $\pi$ and $\sigma$ fixing all elements of $F$ pointwise implies that $\pi(Y)=Y$ and $\sigma(Y)=Y$, which in turn implies that
\begin{equation*}
z:=(y\setminus\{a\})\cup\{b\}=\pi(y)\in\pi(Y)=Y
\end{equation*}
and 
\begin{equation*}
w:=(y\setminus\{a\})\cup\{c\}=\sigma(y)\in\sigma(Y)=Y.
\end{equation*}
Hence $y\bigtriangleup z\bigtriangleup w\in\fs_{\leq 3}(Y)$, while at the same time we have that $|y\bigtriangleup z\bigtriangleup w|=|y|+2$, and therefore $c(y)\neq c(y\bigtriangleup z\bigtriangleup w)$. So $\fs_{\leq 3}(Y)$ is not monochromatic, which shows that $c$ is a bad colouring, and therefore $A$ is \hi{3}-finite.
\end{proof}

Since the set $A$ of atoms in Mostowski's linearly ordered model is B-infinite, the previous proposition allows us to see that \hi{3}-finite does not imply B-finite.

\begin{corollary}
Modulo the theory $\zf$, in general \hi{3}­-finite does not imply B-finite (and consequently, H-finite does not imply B-finite either). Thus, by Corollary~\ref{cor-second-frankel}, the notion of B-finiteness is independent from the notions of H-finiteness and \hi{3}-finiteness.
\end{corollary}

Unfortunately, the set $A$ in this model is \r-infinite for every $n\geq 2$ (this is easily seen with a proof very similar to that of Proposition~\ref{firstfrankelrinfinite}), and so considering this set will not allow us to get any new information about implications between B-finiteness and \r-finiteness, or \hi{2}-finiteness. So we now need to turn our attention to a different family of models.

The models that we will now consider also help address the problem of separating the various notions of \r-finiteness, as $n$ varies. We know that, under specific assumptions on the set $X$, \r-finiteness of $X$ implies \ra{m}-finiteness of $X$ for $m>n$. We were not able to determine if any of these specific assumptions is really necessary for these implications, but we were able to determine that the reverse implications are not provable in $\zf$. To see this, we need to recall the structure of the random graph, sometimes also known as the Rado graph. This is a countable graph $G$ characterized by the fact that for every two disjoint finite sets of vertices $E,F\subseteq G$, there exists a vertex $x\in G$ which is adjacent to every element of $E$ and non-adjacent to every element of $F$. For every $n\geq 2$, we also have the Rado $n$-hypergraph: this is an $n$-hypergraph with set of vertices $G$ (that is, the set of edges is a subset of $[G]^n$) with the property that for every two disjoint finite $E,F\subseteq[X]^{n-1}$, one can find a vertex $x\in G$ such that for every $y\in E$ the set $\{x\}\cup y$ is a vertex, and for every $z\in F$ the set $\{x\}\cup z$ is not a vertex. For a reference where these graphs are studied in some detail, see~\cite[Definition 2.2, Proposition 2.1]{pelayo-ideals} (our Rado $n$-hypergraph is what that author calls ``the $n$-Random graph with two colours''). These hypergraphs are ultrahomogeneous, in the sense that whenever $E,F\subseteq G$ are two finite sets of vertices with an isomorphism between their induced subgraphs, there exists an automorphism of $G$ extending that isomorphism. We use these graphs to define permutation models.

\begin{definition}\label{radomodel}
Start working in a model of $\zfa$ with a countable set of atoms $A$. Let $n\in\mathbb N\setminus\{1\}$. Partition $A$ as a countable union $\bigcup_{m<\omega}A_m$ of countable sets, and equip each of the $A_m$ with the structure of a Rado $n$-hypergraph.
\begin{enumerate}
\item We let $H_n\leq S_A$ be the group of permutations given by
\begin{equation*}
H_n=\{\pi\in S_A\big|(\forall m<\omega)(\pi\upharpoonright A_m\text{ is an automorphism of }A_m)\},
\end{equation*}
(automorphism here means with respect to the structure of Rado $n$-hypergraph).
\item The \textbf{$n$-Rado model} is the class $\hs$ of all hereditarily finite sets with respect to the group $H_n$.
\item The $2$-Rado model will simply be known as the \textbf{Rado model}.
\end{enumerate}
\end{definition}

We begin with a simple proof in the particular case $n=2$, which will set the stage for later proofs with larger $n$.

\begin{proposition}\label{2radoatomsareh2finite}
In the Rado model, the set $A$ of atoms is \hi{2}-finite.
\end{proposition}

\begin{proof}
In the Rado model, we define the colouring $c:[A]^{<\omega}\longrightarrow 2$ by $c(x)=1$ iff there is an $m<\omega$ such that $x\subseteq A_m$ and there are $a,b\in x$ such that $a$ and $b$ are adjacent (according to the structure of the Rado graph that $A_m$ carries). We claim that this colouring witnesses the \hi{2}-finiteness of $A$. To see this, let $Y\subseteq[A]^{<\omega}$ be an arbitrary (hereditarily symmetric) infinite set, and suppose that the finite set $F\subseteq A$ is a support for $Y$. Find a $y\in Y$ such that $y\not\subseteq F$ and pick an $a\in y\setminus F$. Let $m<\omega$ be such that $a\in A_m$ and define $F'=(F\cup y)\cap A_m$. Partition $F'=F_0\cup F_1$, where $F_0$ consists of those atoms in $F'$ that are not adjacent to $a$, and $F_1$ consists of those that are adjacent to $a$. Use the defining property of the Rado graph for the disjoint sets $F_0$ and $F_1\cup\{a\}$ to find an atom $b\in A_m$ which is adjacent to all elements of $F_1$, as well as to $a$, and not adjacent to any element of $F_0$; similarly use the same property applied to the sets $F_0\cup\{a\}$ and $F_1$ to find an atom $c\in A_m$ which is adjacent to all elements of $F_1$ and not adjacent to any element of $F_0$ nor to $a$. Note that our choice of $b$ and $c$ ensure that the subgraph of $A_m$ induced by the set of vertices $F'\cup\{a\}$ is isomorphic to that induced by the set of vertices $F'\cup\{b\}$ via an isomorphism fixing $F'$; by ultrahomogeneity this implies that there is an automorphism of $A_n$ extending this isomorphism and consequently there exists a $\pi\in H_2$ such that $\pi$ is the identity on all $A_k$, for $k\neq m$, and $\pi\upharpoonright A_m$ is an automorphism fixing $F'$ pointwise and mapping $a$ to $b$. With an entirely analogous argument we obtain an element $\sigma\in H_2$ which is the identity on $A_k$ for $k\neq m$ and such that $\sigma\upharpoonright A_m$ is an automorphism fixing $F'$ pointwise and mapping $a$ to $c$. Hence both $\pi$ and $\sigma$ fix $F$ pointwise and thus they fix $Y$ setwise, which implies that $Y$ contains both $z=\pi(y)=(y\setminus\{a\})\cup\{b\}$ and $w=\sigma(y)=(y\setminus\{a\})\cup\{c\}$. Therefore both $\{a,b\}=y\bigtriangleup z$ and $\{a,c\}=y\bigtriangleup w$ belong to $\fs_{\leq 2}(Y)$, but notice that (since $b$ is adjacent to $a$ but $c$ is not) we have by construction that $c(\{a,b\})=1\neq 0=c(\{a,c\})$. This shows that $\fs_{\leq 2}(Y)$ cannot be monochromatic, and hence $A$ is \hi{2}-finite, which finishes the proof.
\end{proof}

This is the first example that we have exhibited of an infinite \hi{2}-finite set.

\begin{corollary}
Modulo the theory $\zf$, \hi{2}-finite does not imply finite.
\end{corollary}

Thus, in the Rado model, the set of atoms is also \ra{2}-finite, by Theorem~\ref{h2finimpliesr2fin}. The next proposition shows that something more general is true in all of the $n$-Rado models.

\begin{proposition}\label{nradomodel}
Let $n\geq 2$. In the $n$-Rado model, the set $A$ of atoms is \ra{k}-finite for every $k\geq n$. Moreover, if $n\geq 3$, then the set $A$ is \ra{k}-infinite for all $k<n$.
\end{proposition}

\begin{proof}
Let $k\geq n$. To see that $A$ is \ra{k}-finite, we will essentially use the same colouring as in Proposition~\ref{2radoatomsareh2finite}. In other words, we let $c:[A]^k\longrightarrow 2$ be given by $c(x)=1$ if and only if for some $m<\omega$, $x\subseteq A_m$ and there is a $y\in[x]^n$ such that $y$ is an $n$-hyperedge of the $n$-hypergraph at $A_m$. We will show that for no infinite $X\subseteq A$ can we have that $[X]^n$ is monochromatic. To see this, assume that we have such a set $X$ within our model, say with support $F\subseteq A$. Choose an $a\in X\setminus F$ and let $m<\omega$ be such that $a\in A_m$. Now let $F'=F\cap A_m$. We recursively build atoms $a=a_1,a_2,\ldots,a_k\in A_m$ and $a=b_1,b_2,\ldots,b_k\in A_m$ as follows: suppose that we already know $a_1,\ldots,a_l$ and $b_1,\ldots,b_l$ for $l<k$. We let $F_0=\{x\in [F']^{n-1}\big|x\cup\{a\}\text{ is not a hyperedge}\}$ and $F_1=\{x\in [F']^{n-1}\big|x\cup\{a\}\text{ is a hyperedge}\}$. We take any $a_{l+1}\in A_m$ such that $\{a_{l+1}\}\cup x$ is not a hyperedge for all $x\in F_0\cup[\{a_1,\ldots,a_l\}]^{n-1}$ and $\{a_{l+1}\}\cup x$ is a hyperedge for all $x\in F_1$; analogously we take any $b_{l+1}$ such that $\{b_{l+1}\}\cup x$ is not a hyperedge for all $x\in F_0$ and $\{b_{l+1}\}\cup x$ is a hyperedge for all $x\in F_1\cup[\{b_1,\ldots,b_l\}]^{n-1}$. This construction ensures two things:
\begin{itemize}
\item for every $l\leq k$, the subgraph of $A_m$ induced by the vertex set $F'\cup\{a\}$ is isomorphic to that induced by $F'\cup\{a_l\}$ and also to that induced by $F'\cup\{b_l\}$, in both cases via an isomorphism that fixes $F'$;
\item the subgraph of $A_m$ induced by the vertex set $\{a_1,\ldots,a_k\}$ is independent (i.e. it has no hyperedges) whereas that induced by $\{b_1,\ldots,b_k\}$ is complete (i.e. every $n$-sized subset is a hyperedge).
\end{itemize}
Consequently, we have that $c(\{a_1,\ldots,a_k\})=0$ and $c(\{b_1,\ldots,b_k\})=1$ (notice that the hypergraph induced by $\{b_1,\ldots,b_k\}$ will have at least one hyperedge because $n\leq k$), so as soon as we can show that each $a_l$ and each $b_l$ are elements of $X$, this will prove that $[X]^k$ is not monochromatic for $c$. The first bullet point above ensures that there are $\pi_l,\sigma_l\in H_n$, for $1\leq l\leq k$, such that $\pi_l$ fixes pointwise all $A_t$ with $t\neq m$ and $\pi_l\upharpoonright A_m$ is an automorphism fixing $F'$ and mapping $a$ to $a_l$, and similarly $\sigma_l$ fixes pointwise all $A_t$ with $t\neq m$ and $\sigma_l\upharpoonright A_m$ is an automorphism fixing $F'$ and mapping $a$ to $b_l$. Hence each of the $\pi_l$ and $\sigma_l$ fix $F$ pointwise, and thus $\pi_l(X)=X=\sigma_l(X)$, which implies that $X$ contains $a_l=\pi_l(a)$ and $b_l=\sigma_l(a)$ as elements. This finishes the proof that $[X]^k$ cannot be monochromatic for $c$, and therefore the set $A$ is \ra{k}-finite.

Now for the ``moreover'' part, assume that $n\geq 3$ and let $k<n$. Suppose that we have a colouring $c:[A]^k\longrightarrow 2$ within the model, then this colouring has a finite support $F\subseteq A$. Let $m$ be large enough that $A_m\cap F=\varnothing$. If we have any two $k$-element subsets of $A_m$, $x$ and $y$, notice that the corresponding induced subgraphs are isomorphic ---both hypergraphs contain no edges since they have less than $n$ vertices---. Hence by ultrahomogeneity, there is an automorphism of $A_m$ mapping $x$ to $y$, and hence a permutation $\pi$ of $A$ extending this automorphism and fixing all $A_l$ for $l\neq m$. In particular $\pi$ fixes each element of $F$ and therefore we have that $(y,c(x))=(\pi(x),\pi(c(x)))=\pi((x,c(x)))\in\pi(c)=c$, meaning that $c(y)=c(x)$. This shows that the infinite subset $A_m$ of $A$ is monochromatic, and therefore $A$ is \ra{k}-infinite. This finishes the proof.
\end{proof}

\begin{corollary}
Modulo the theory $\zf$, for every $2\leq n<m$ we have that, in general, \ra{m}-finite does not imply \r-finite. 
\end{corollary}

In view of the previous corollary, it is natural to ask whether it is possible to have an implication from \ra{m}-finite to \ra{n}-finite, when $n<m$, in the presence of some extra hypotheses. For example, Proposition~\ref{prop:rn-to-rn+1} shows that the reverse implication holds for linearly orderable sets; it is therefore natural to ask whether \ra{m}-finite implies \ra{n}-finite, for $n<m$, when assuming linear orderability of the relevant sets. We finish this section by answering this question in the negative; in fact, we will show the much stronger statement that \ra{m}-finite does not imply \ra{n}-finite even if one assumes the Boolean Prime Ideal theorem. Recall that the Boolean Prime Ideal Theorem (form 14 in Howard and Rubin's book~\cite{howard-rubin}) states that every Boolean algebra has a maximal (equivalently, prime) ideal. This weak choice principle is equivalent to various other well-known statements (such as the ultrafilter theorem, Tychonoff's theorem for Hausdorff spaces, the compactness theorem for propositional logic, the Erd\H{o}s--de Bruijn theorem, among others) and it implies---and is therefore stronger than---the statement that every set can be linearly ordered. We will henceforth denote the Boolean Prime Ideal theorem by $\bpi$.

In order to get a model where $\bpi$ holds, we consider the {\bf linearly ordered Rado $n$-hypergraph}. This is a countably infinite structure equipped both with a linear order (with respect to which the structure is isomorphic to $(\mathbb Q,\leq_{\mathbb Q})$) and at the same time with an $n$-hypergraph structure; in this context, an automorphism of the structure is understood to be a bijection that respects both the linear order and the hypergraph structure. The linearly ordered Rado $n$-hypergraph can be thought of as the Fra\"{\i}ss\'e limit of the class of all finite linearly ordered $n$-hypergraphs, as described in~\cite[p. 110--111]{kechris-pestov-todorcevic}. The first key property of this graph $G_n$ is that, given any two disjoint finite sets $F_0,F_1\subseteq[G_n]^{n-1}$, it is always possible to find a vertex $v\in G_n$ such that $\{v\}\cup x$ is not a hyperedge for $x\in F_0$ and $\{v\}\cup x$ is a hyperedge for $x\in F_1$, and furthermore the vertex $x$ can be found in any desired previously prescribed position, with respect to the linear order, in relation to the vertices from $F_0\cup F_1$. The second main property is ultrahomogeneity: if $F_0,F_1$ are two finite sets of vertices such that there is an (order-preserving) isomorphism between the corresponding induced subgraphs, then there is an (order-preserving) automorphism of $G_n$ extending the original finite isomorphism.

The following definitions are entirely analogous to Definition~\ref{radomodel}, except we now take into account the additional linear order structure.

\begin{definition}\label{def-lo-rado-model}
Work in a model of $\zfa$ with a countable set of atoms $A$, and let $n\in\mathbb N\setminus\{1\}$. Partition $A$ as a countable union $\bigcup_{m<\omega}A_m$ of countable sets, and equip each of the $A_m$ with the structure of a linearly ordered Rado $n$-hypergraph.
\begin{enumerate}
\item The group of permutations $H_n'\leq S_A$ is defined as follows:
\begin{equation*}
H_n'=\{\pi\in S_A\big|(\forall m<\omega)(\pi\upharpoonright A_m\text{ is an automorphism of }A_m)\}
\end{equation*}
(by automorphism of $A_m$ we mean an automorphism with respect to both the linear order and the hypergraph structure).
\item The \textbf{linearly ordered $n$-Rado model} is the class $\hs$ of all hereditarily finite sets with respect to the group $H_n'$.
\end{enumerate}
\end{definition}

The model just defined bears significant similarities with the ones from Definition~\ref{radomodel}. In particular, we immediately get the following proposition.

\begin{proposition}
Let $n\geq 2$. In the linearly ordered $n$-Rado model, the set $A$ of atoms is \ra{k}-finite for every $k\geq n$ and, if $n=2$, then $A$ is \hi{2}-finite, whereas if $n\geq 3$ then $A$ is \ra{k}-infinite for all $k<n$.
\end{proposition}

\begin{proof}
Exactly as in the proofs of Propositions~\ref{nradomodel} and~\ref{2radoatomsareh2finite}.
\end{proof}

Now, the key reason why we introduced these models is contained in the following theorem.

\begin{theorem}\label{linearlyorderedmodel}
For each $n\in\mathbb N\setminus\{1\}$, the linearly ordered $n$-Rado model satisfies $\bpi$.
\end{theorem}

\begin{proof}
The automorphism group $\aut(G_n)$ of the linearly ordered $n$-Rado hypergraph is extremely amenable, as seen in~\cite[p. 110-111]{kechris-pestov-todorcevic}; the groups $H_n'$ used in Definition~\ref{def-lo-rado-model} are isomorphic to a direct product of countably many copies of $\aut(G_n)$ and so, by~\cite[Lemma 6.7 (iii)]{kechris-pestov-todorcevic} these groups are extremely amenable too. The fact that the linearly ordered $n$-Rado model is defined using finite supports means that, in the terminology of~\cite[Section 4.2]{jech-choice}, the normal filter of subgroups we are using is precisely the filter of open subgroups of $H_n'$ (viewed as a subgroup of $S_\omega$ with the pointwise convergence topology). Therefore, by~\cite[Theorems~5.1 and~5.2]{ramseyactions}, the Fraenkel--Mostowski model that we obtain (namely the linearly ordered $n$-Rado model) satisfies $\bpi$ plus the negation of the Axiom of Choice.
\end{proof}

Theorem~\ref{linearlyorderedmodel} has two important corollaries regarding independence in $\zf$.

\begin{corollary}\label{bpi-model}
Modulo the theory $\zf+\bpi$, for every $2\leq n<m$ we have that, in general, \ra{m}-finite does not imply \r-finite. In particular, \ra{m}-finite does not imply \r-finite even for linearly orderable sets.
\end{corollary}

\begin{proof}
Recall that a formula $\varphi(x_1,\ldots,x_n)$ is {\em boundable}, in the sense of Jech and Sochor~\cite{jech-sochor}, if there is an ordinal $\alpha$ (more generally, some absolute ordinal-valued term that depends on $x_1,\ldots,x_n$) such that for all $x_1,\ldots,x_n$, $\varphi(x_1,\ldots,x_n)$ holds if and only if $\wp^\alpha(x_1\cup\cdots\cup x_n)\vDash\varphi(x_1,\ldots,x_n)$; a {\em boundable statement} is the existential closure of a boundable formula. The statement ``there exists a set that is \ra{m}-finite and \r-infinite'' is boundable, and in particular it is injectively boundable as defined by Pincus~\cite[2A5, p. 736]{pincus-transfer}. Thus by a general transfer theorem of Pincus (stated by pieces in~\cite[Theorem 4 and note in p. 145]{pincus-add-dep-choice}, see also~\cite[p. 547]{pincus-add-dep-choice-to-pi}; for a full statement of the transfer theorem see~\cite[p. 286]{howard-rubin}), the conjunction of this statement with $\bpi$ is transferable. In other words, from the existence of a Fraenkel--Mostowski model satisfying $\bpi$ together with the existence of a \ra{m}-finite and \r-infinite set, the existence of a model of $\zf$ satisfying the same statement follows.
\end{proof}

\begin{corollary}
Modulo the theory $\zf$, in general neither \hi{2}-finite nor \r-finite (for any $n\geq 2$) imply B-finite. Consequently, the notion of B-finite is independent from each of the notions of \r-finite ($n\geq 2$) and \hi{2}-finite.
\end{corollary}

\begin{proof}
The models of $\zf$ obtained in Corollary~\ref{bpi-model} all satisfy $\bpi$ and, in particular, every set can be linearly ordered in these models. Thus every infinite set is B-infinite in these models, and so neither \r-finite nor \hi{2}-finite imply B-finite; now just invoke Corollary~\ref{conclusionsfirstfrankel} (2) for the unprovability of the reverse implications.
\end{proof}

\section{Open questions}

We would like to close this paper by mentioning a few questions that are not answered here, but naturally suggest themselves after the results obtained.

\begin{question}\hfill
\begin{enumerate}
\item Is there a model of $\zfa$ with a set which is H-finite but \hi{3}-infinite?
\item Does \ra{n}-finite imply \ra{m}-finite whenever $n<m$, in the absence of any further assumptions (or at least with weaker assumptions than the linear orderability) about the relevant set?
\item Given $n\geq 3$, is it the case that \hi{2}-finite implies \r-finite?
\item How do the notions of H-finite and its variations, or of \r-finite, change if one allows for colourings with more than two colours in the definition?
\item What interesting things can one say about the analogous notions of finiteness arising from Gowers's $\mathrm{Fin}_k$ theorem?
\end{enumerate}
\end{question}

\section*{Acknowledgements}

Most of the work in this paper was carried out during the summer of 2018, as part of the REU (Research Experience for Undergraduates) program at the Department of Mathematics, University of Michigan, where the first two authors were advised by the third. Within this context, the first author was supported by NSF grant number DMS-1501625, whereas the second author was supported by the University of Michigan Mathematics Department Strategic Fund; afterwards, during the final stages of preparation of this paper, the third author was supported by grant FORDECYT \# 265667 from the Consejo Nacional de Ciencia y Tecnolog\'{\i}a. The authors are also grateful to Lorenzo Carlucci for reading a previous version of this paper and pointing out some useful references, as well as to Fernando N\'u\~nez Rosales for pointing out that the linearly ordered $n$-Rado model satisfies $\bpi$ (in an earlier version of this paper, we only claimed that this model satisfied the weaker Linear Ordering principle, which we tortuously proved using ideas from~\cite{mostowski-linear-orders} rather than the much cleaner argument from this version).

\end{document}